\documentclass[12pt,reqno]{amsart}
\usepackage{amsmath}
\usepackage{amssymb}
\usepackage{mathrsfs}
\usepackage{amstext}
\usepackage{a4wide}
\usepackage{graphicx}
\allowdisplaybreaks \numberwithin{equation}{section}
\usepackage{color}
\usepackage{cases}

\numberwithin{equation}{section}

\newtheorem{theorem}{Theorem}[section]

\newtheorem{lemma}[theorem]{Lemma}

\theoremstyle{definition}

\newtheorem{definition}[theorem]{Definition}

\theoremstyle{remark}

\begin{document}

\title
{Desingularization of steady vortex of perturbation type in the lake equations}

\author{Daomin Cao, Jie Wan, Changjun Zou}

\address{Institute of Applied Mathematics, Chinese Academy of Sciences, Beijing 100190, and University of Chinese Academy of Sciences, Beijing 100049,  P.R. China}
\email{dmcao@amt.ac.cn}

\address{Institute of Applied Mathematics, Chinese Academy of Sciences, Beijing 100190, and University of Chinese Academy of Sciences, Beijing 100049,  P.R. China}
\email{wanjie15@mails.ucas.edu.cn}

\address{Institute of Applied Mathematics, Chinese Academy of Sciences, Beijing 100190, and University of Chinese Academy of Sciences, Beijing 100049,  P.R. China}
\email{zouchangjun17@mails.ucas.ac.cn}

%\thanks{This work is partially supported by ARC}

\begin{abstract}
	In this paper, we study the desingularization of steady lake model of perturbation type with general nonlinearity $f$. Using the modified vorticity method, we construct a family of steady solutions with vanishing circulation, which constitute a desingularization of a singular vortex. The localization of the singular vortex is determined only by the vanishing rate of the circulation. Some qualitative and asymptotic properties are also established.
\end{abstract}

\maketitle

\section{Introduction}
In this paper, we study the desingularization of steady vortex in the lake equations. When Froude number vanishes, the lake equations in a planar domain $D$ with prescribed initial and boundary condition are
\begin{numcases}
{}
\label{1-1} \text{div}(b\,\mathbf{v})=0\,\ \, \ \ \ \ \ \ \  \ \, &\text{in}\  $\mathbb{R}_+\times D$,  \\
\label{1-2} \partial_t\mathbf{v}+(\mathbf{v}\cdot\nabla)\mathbf{v}=-\nabla P \ \ \ \  &\text{in}\  $\mathbb{R}_+\times D$,\\
\label{1-3}  b\mathbf{v}\cdot \mathbf{n}=\nu\ \ \ \ \ \ \ \ \ \ \ \ \ \ \, &\text{on}\  $\mathbb{R}_+\times \partial D$,\\
\label{1-4}  \mathbf{v}(0,x)=\mathbf{v}_0(x)\ \ \ \ \, &\text{in}\ $D$,
\end{numcases}
where   $D\in\mathbb{R}^2$ is a bounded simply-connected domain, $b:\bar{D}\to\mathbb{R}$ is a nonnegative depth function, $\mathbf{v}=(v_1,v_2)$ is the velocity field, $P$ is the scalar pressure, $\mathbf{n}$ is the unit outward normal to $\partial D$, and $\nu$ is a penetration condition defined on $\partial\Omega$ which satisfies compatibility condition $\int_{\partial D}\nu d\sigma=0$. It is reasonable to assume $b$ can only attain $0$ on $\partial D$, which is the border of lake. For the case $\nu\equiv 0$, which is usually called the impermeability boundary condition, we refer to \cite{CZZ}. In this paper, we consider the case $\nu\not\equiv 0 $, namely, the steady vortex is added as small perturbation to a nontrivial irrotational background flow.

When the typical velocity magnitude is small in comparison to the magnitude of gravity waves, the incompressible 3D Euler equations can be approximated by the lake equations. Therefore, lake equations apply to a domain which is shallow compared to its width and whose free surface exhibits negligible surface motion. It is noteworthy that the two-dimensional Euler equations and the three-dimensional axisymmetric Euler equations are two particular cases of this lake model. For more background on lake model, we refer to \cite{D1}.

By virtue of \eqref{1-1} and boundary condition, there exists a Stokes stream function $\psi$ such that
\begin{equation}\label{1-5}
\mathbf{v}=\left(\frac{\partial_2\psi}{b},-\frac{\partial_1\psi}{b}\right).
\end{equation}\label{1-6}
Define the vorticity of the flow $\omega=\text{curl}\mathbf{v}:=\partial_1v_2-\partial_2v_1$,
using the vector identity 
\begin{equation}
\frac{1}{2}\nabla(|\mathbf{v}|^2)= (\mathbf{v}\cdot\nabla)\mathbf{v}+\mathbf{v}\times(\nabla\times\mathbf{v}),
\end{equation}
we have
\begin{equation}\label{1-7}
\partial_t\mathbf{v}+\frac{1}{b}\omega\nabla\psi=-\nabla(P+\frac{1}{2}|\mathbf{v}|^2).
\end{equation}
If we apply $curl$ on \eqref{1-7}, one obtains the following vorticity equation
\begin{equation}\label{1-8}
\partial_t \omega+div(\omega\mathbf{v})=0.
\end{equation}
We define the potential vorticity to be $\zeta:=b^{-1}\omega$. According to \eqref{1-1}, equation \eqref{1-8} becomes
\begin{equation}\label{1-9}
\partial_t\zeta+\mathbf{v}\cdot\nabla\zeta=0,
\end{equation}
which is a nonlinear transport equation. Now we consider the steady lake equations. By substituting \eqref{1-5} into \eqref{1-9}, we have
\begin{equation}\label{1-10}
\nabla^{\perp}\psi\cdot\nabla\zeta=0,
\end{equation}
where $x^{\perp}=(x_2,-x_1)$ denotes clockwise rotation through $\frac{\pi}{2}$. Equation \eqref{1-10} suggests
that $\psi$ and $\zeta$ are functionally dependent. As a result, if we take $\zeta=f(\psi)$ for some vorticity function $f:\mathbb{R}\to\mathbb{R}$, equation \eqref{1-10} automatically holds. Define $F$ to be a primitive function of $f$, we obtain $\nabla F(\psi)=\zeta\nabla\psi$. So once we find $\psi$, the velocity of the flow is given by \eqref{1-5} and the pressure is given by \eqref{1-7}, namely $P=-F(\psi)-\frac{1}{2}|\mathbf{v}|^2$. 
By \eqref{1-1} and the definition of $\zeta$, there holds
\begin{equation}\label{1-11}
-\text{div}(b^{-1}\nabla \psi)=b\zeta.
\end{equation}
Meanwhile, the boundary condition \eqref{1-3} and the definition of $\psi$ \eqref{1-5} lead to
\begin{equation}\label{1-12}
\nabla^{\perp} \psi\cdot\mathbf{n}=\nu \ \ \ \ \text{on}  \ \ \partial D.
\end{equation} 
The system \eqref{1-10}, \eqref{1-11} and \eqref{1-12} is called the vortex formulation of the lake equations. In the following, we study desingularization of this system
\begin{equation}\label{1-13}
\begin{cases}
-\text{div}(b^{-1}\nabla \psi)=b\zeta\ \ \ \  &\text{in}  \ \ D,\\
\nabla^{\perp}\psi\cdot\nabla\zeta=0\ \ \ \ &\text{in}  \ \ D,\\
\nabla^{\perp} \psi\cdot\mathbf{n}=\nu\ \ \ \ &\text{on}  \ \ \partial D. \\
\end{cases}
\end{equation}
In order to solve this system, as in \cite{DV}, we assume $q\in C^1(\bar{D})\cap C^2(D)$ be the solution of
\begin{equation}\label{1-14}
\begin{cases}
-\text{div}(b^{-1}\nabla q)=0\ \ \ \  &\text{on}  \ \ D,\\
\nabla^{\perp} q\cdot\mathbf{n}=\nu\ \ \ \ &\text{on}  \ \ \partial D. \\
\end{cases}
\end{equation}
Notice that $q$ corresponds to a nontrivial irrotational background flow. If we define $u:=\psi-q$ to be the Stokes stream function of vortex perturbation, $u$ will satisfy
\begin{equation}\label{1-15}
\begin{cases}
-\text{div}(b^{-1}\nabla u)=b\zeta\ \ \ \  &\text{in}  \ \ D,\\
\nabla^{\perp}(u+q)\cdot\nabla\zeta=0\ \ \ \ &\text{in}  \ \ D,\\
\nabla^{\perp} u\cdot\mathbf{n}=0\ \ \ \ &\text{on}  \ \ \partial D. \\
\end{cases}
\end{equation}
Similar to results of Burton \cite{B}, it suffices to find $(u^\varepsilon,\zeta^\varepsilon)$ satisfying $\zeta^\varepsilon=f(u^\varepsilon)$ for some increasing function $f$. Thus we are interested in studying the asymptotics of solutions of
\begin{equation}\label{1-16}
\mathcal{L}u^\varepsilon=-\frac{1}{b}\,\text{div}\frac{\nabla u^\varepsilon}{b}=\frac{\delta(\varepsilon)}{\varepsilon^2}f(u^\varepsilon+q)=\zeta^\varepsilon.
\end{equation}
The total circulation of vortex perturbation is defined by
\begin{equation}\label{1-17}
\mathcal{C}(\varepsilon):=\int_D\zeta^\varepsilon b=\kappa_0\delta(\varepsilon),
\end{equation}
where $\kappa_0>0$ is a fixed constant and $\delta(\varepsilon)$ is the same term in \eqref{1-16} dominating the vanishment of perturbation. In this paper, we are concerned with weak solution of \eqref{1-16} and \eqref{1-17} as the circulation $\mathcal{C}(\varepsilon)$ vanishes, namely, $\delta(\varepsilon)$ is of order $o_\varepsilon(1)$. 

Our work originates from the study of the two-dimensional Euler equations.  In \cite{YJ}, Yang considered the Euler flow in $\mathbb{R}^+_2$ with $q(x)=Wx_1+d$ and constructed a class of vortices with vanishing circulation. He has proved that as $\varepsilon\to 0^+$, these vortices will shrink, and concentrate near some point on $\partial \mathbb{R}^+_2$, where $q$ attains its maximum.  In \cite{LYY}, Li-Yan-Yang obtained a similar result in bounded domain with the vortices localized around the maximum point of $q$. It is notable that in \cite{YJ} and \cite{LYY}, the circulation vanishes at an asymptotic rate $\delta(\varepsilon)\approx 1/\ln\frac{1}{\varepsilon}$. 

Inspired by an observation in Berger-Fraenkel \cite{BF2}, Didier Smets-Van Schaftingen \cite{SV} studied the case $\delta(\varepsilon)\approx 1$ and constructed vorices concentrating near a critical point of Kirchhoff-Routh function with $q\equiv0$. The method in \cite{YJ}, \cite{LYY} and \cite {SV} was based on the mountain pass theorem of Ambrosetti and Rabinowitz \cite{AR}, which is now called the stream-function method. 

In \cite{CW3}, Cao-Wang-Zhan studied the case $\delta(\varepsilon)$ is sharp by using a different strategy. Their method is called the vorticity method, which was introduced by Arnold \cite{A} \cite{AK} and implemented successfuly by Turkington \cite{T}. They proved that the limiting behavior of vortices in this sharp-vanishing case is similar to that in \cite{LYY}. 

For the lake model, there are less related work. Using the stream-function method, De Valeriola-Van Schaftingen \cite{DV} constructed a family of desingularized solutions of equation \eqref{1-13} when $f(s)=s^p_+$ for some $p>1$ and $\delta(\varepsilon)\approx 1$. Dekeyser \cite{D1} \cite{D2} also investigated the problem by the vorticity method. However, the vorticity distribution function $f$ in \cite{D1} \cite{D2} is left undetermined. Recently, Cao-Zhan-Zou \cite{CZZ} studied the case $\delta(\varepsilon)=1$ with general nonlinearity $f$.   

Now here comes the question: can we obtain a different asymptotic localization of the vortices by adjusting the vanishing rate $\delta(\varepsilon)$? In this paper, we give a positive answer. Actually, we give a classification result: when the total circulation of flow tends to zero at a different rate (dominated by $\delta(\varepsilon)$), the asymptotic localization of vortices can be one of the three following positions: the maximum point of the depth function $b$, the maximum point of $q$, or the maxmum point of a function which is a linear combination of $b$ and $q$. 

By an adaption of the vorticity method, we study the desingularization of \eqref{1-16} and \eqref{1-17}. The vorticity function in this paper is of general type. In fact, we can prove that the assumption on $f$ can be reduced when the circulation $\mathcal{C}(\varepsilon)$ vanishes at a higher rate. More precisely, we make the following assumptions on $f$:
\begin{itemize}
	\item[(H1)] $f$ is continuous on $\mathbb{R}\setminus\{0\}$, $f(s)=0$ for $s\leq 0$ and $f$ is strictly increasing in $[0,+\infty)$ with $f(0^+):=\lim_{s\to0^+}f(s)\ge0$;
	\item[(H2)] There exists $\vartheta_0\in(0,1)$ such that
	$$\int_0^s(f(r)-f(0^+))dr\leq \vartheta_0(f(s)-f(0^+))s,\,\,\forall\,s\geq0.$$
\end{itemize}

Throughout the sequel we shall use the following notations: $B_r(y)$ denotes an open ball in $\mathbb{R}^2$ of center $y$ and radius $r>0$; $\chi_{_A}$ denotes the characteristic function of $A\subseteq D$; Lebesgue measure on $\mathbb{R}^2$ is denoted $\textit{m}$, and is to be understood as the measure defining any $L^p$ space and $W^{1,p}$ space, except when stated otherwise; $\nu$ denotes the measure on $\mathbb{R}^2$ having density $b$ with respect to $\textit{m}$ and $|\cdot|$ denotes the $\nu$-measure; $O_\varepsilon(1)$ denotes a number which stays bounded as $\varepsilon$ goes to zero and $o_\varepsilon(1)$ denotes a number which goes to zero as $\varepsilon$ goes to zero. $O_\varepsilon(1)$ and $o_\varepsilon(1)$ only depend on $\varepsilon$. In order to study weak solution of \eqref{1-16} and \eqref{1-17}, we define $\mathcal{K}$ as the inverse of $\mathcal{L}$ as follows. One can check the operator $\mathcal{K}$ is well-defined, see, e.g., \cite{D1} \cite{DV}.
\begin{definition}\label{def1}
	The Hilbert space $H(D)$ is the completion of $C_0^\infty(D)$ with the scalar products
	\begin{equation*}
	\langle u,v\rangle_H=\int_D\frac{1}{b^2}\nabla u\cdot\nabla v d\nu.
	\end{equation*}
	We define inverses $\mathcal{K}$ of $\mathcal{L}$ in the weak solution sense,
	\begin{equation}\label{1-18}
	\langle \mathcal{K}u,v\rangle_H=\int_D uv d\nu \ \  \text{for all}\  v\in H(D), \ \ when \ u \in  L^p(D)\ \text{for some}\ p>1.
	\end{equation}
\end{definition}
As in \cite{D1}, we introduce the following definition.
\begin{definition}\label{def2}
	 A lake $(D,b)$ is said to be continuous if the operator $\mathcal{K}$ admits the following integral kernel representation:
	 $$\mathcal{K}\zeta(x)=b(x)\int_D G(x,y)\zeta(y)d\nu(y)+\int_D R(x,y)\zeta(y)d\nu(y) \ \ \ \forall\zeta\in L^p(D,d\nu),p>1$$
	 where $G$ is the Green’s function for $-\Delta$ with Dirichlet boundary conditions, and $R: D\times D \to\mathbb{R}$ is a bounded and measurable correction function.
\end{definition}
More general concept of continuous lake is defined in \cite{D2}. We remark that this class covers two situations encountered in the literature: first, the case of $b\in W^{1,\infty}(D)\cap C^1_{\text{loc}}(D)$ with additional condition that $\text{inf}_{D}b>0$; second, the case of a degenerated depth function that vanishes as a polynomial of some regularized distance at the boundary. Mixed conditions are also allowable. For more examples, we refer to \cite{D1}.

In this paper, we assume $D\subset \mathbb{R}^2$ be a bounded simply-connected Lipschitz domain and nonnegative depth function $b(x)\in C^{\alpha}(\bar{D})\cap C^{1}(D)$ for some $\alpha\in(0,1)$ satisfying $b(x)>0$ in $D$. A continuous lake $(D,b)$ is called a continuous regular lake if $D$ and $b$ satisfy above assumptions. To clarify the first theorem, we denote $\mathcal{S}:=\{x\in\bar{D}\ |\ b(x)=\text{max}_{\bar{D}} b\}$.

\begin{theorem}\label{thm1}
	Let $(D,b)$ be a continuous regular lake, $f$ be a function satisfying (H1) and (H2). As $\varepsilon$ goes to zero, impose following restrictions on $\delta(\varepsilon)$:
	\begin{itemize}
		\item[(a)] If $\mathcal{S}\cap D\not=\varnothing$, let $\delta(\varepsilon)=o_\varepsilon(1)$ and $(\ln\frac{1}{\varepsilon})^{-1}/\delta(\varepsilon)=o_\varepsilon(1)$.
		\item[(b)] If $\mathcal{S}\subset \partial D$, let $\delta(\varepsilon)=o_\varepsilon(1)$ and $(\ln\ln\frac{1}{\varepsilon}/\ln\frac{1}{\varepsilon})/\delta(\varepsilon)=o_\varepsilon(1)$.
	\end{itemize}
 Then for all sufficiently small $\varepsilon>0$, there exists a family of solutions $(\psi^\varepsilon,\zeta^\varepsilon)$ with the following properties:
	\begin{itemize}
		\item[(i)]For any $p>1$, $\zeta^\varepsilon\in L^p(D)$, $\psi^\varepsilon\in W^{2,p}_{\text{loc}}(D)$ and satisfies
	    \begin{equation*}	
	    \mathcal{L}\psi^\varepsilon=\zeta^\varepsilon\ \ \text{a.e.} \ \text{in} \ D.
	    \end{equation*}
	    $(\psi^\varepsilon,\zeta^\varepsilon)$ is of the form
	    \begin{equation*}
	    \psi^\varepsilon=\mathcal{K}\zeta^\varepsilon+q-\mu^\varepsilon,\ \  \zeta^\varepsilon=\frac{\delta(\varepsilon)}{\varepsilon^2}f(\psi^\varepsilon),\ \
	    \int_D\zeta^\varepsilon d\nu=\kappa_0\delta(\varepsilon)
	    \end{equation*}
	    for some $\mu^\varepsilon$ depending on $\varepsilon$. 
		\item[(ii)] 
		One has $$\lim\limits_{\varepsilon\to0^+}\frac{\ln diam(supp(\zeta^\varepsilon))}{\ln\varepsilon}=1.$$ Furthermore, as $\varepsilon$ goes to zero, $\text{supp}(\zeta^\varepsilon)$ will shrink to $\mathcal{S}$. That is, for every $l>0$, there holds $$supp{\zeta^\varepsilon}\subset \mathcal{S}_l:=\{x\in D|dist(x,\mathcal{S})<l\}$$ provided $\varepsilon>0$ is sufficient small. 
		\item[(iii)]
		If $\mathcal{S}\cap D\not=\varnothing$, there exists some constant $\eta>0$ such that $$dist(supp(\zeta^\varepsilon),\partial D)>\eta.$$ One has
		\begin{equation*}
		\mu^\varepsilon= \frac{\kappa_0 sup_D b}{2\pi}\delta(\varepsilon)\ln \frac{1}{\varepsilon}+O_\varepsilon(1).
		\end{equation*}
		If $\mathcal{S}\subset \partial D$, there exist constants $C_d,\gamma_0>0$ such that $$dist(supp(\zeta^\varepsilon),\partial D)\ge\frac{C_d}{(\ln\frac{1}{\varepsilon})^{\gamma_0}}.$$  
		One has
		\begin{equation*}
		\mu^\varepsilon= \frac{\kappa_0 sup_D b}{2\pi}\delta(\varepsilon)\ln \frac{1}{\varepsilon}- O_\varepsilon(1)\cdot\delta(\varepsilon)\cdot\ln\ln\frac{1}{\varepsilon}+O_\varepsilon(1).
		\end{equation*}
		\item[(iv)]
		Let the center of vorticity be
		\begin{equation*}
		X^\varepsilon=\frac{\int_D x\zeta^\varepsilon(x)dm(x)}{\int_D \zeta^\varepsilon(x)dm(x)},
		\end{equation*}
		and define the rescaled version of $\zeta^\varepsilon$ to be
		\begin{equation*}
		\xi^\varepsilon(x)=\frac{\varepsilon^2}{\delta(\varepsilon)}\zeta^\varepsilon(X^\varepsilon+\varepsilon x).
		\end{equation*}
		Then every accumulation points of $\xi^\varepsilon(x)$ as $\varepsilon \to 0^+$, in the weak topology of $L^2$, are radially nonincreasing functions.
	\end{itemize}
\end{theorem}
Here, the condition $\mathcal{S}\cap D\not=\varnothing$ means there exists some deepest point inside the lake, while $\mathcal{S}\subset \partial D$ means the deepest point is only on the boundary. When the vanishment is gentler than the critical value $\delta(\varepsilon)=1/\ln\frac{1}{\varepsilon}$, Theorem \ref{thm1} asserts $supp(\zeta^\varepsilon)$ will be located around the deepest position in the lake. For the first case, the vortices will be far away from the boundary; for the second although the vortices will concentrate near the boundary, there must be a small positive distance between $supp(\zeta^\varepsilon)$ and $\partial D$.

When the vanishing rate $\delta(\varepsilon)$ is of the critical value $1/\ln\frac{1}{\varepsilon}$, the limiting behavior of solutions will be somehow different. To illustrate it, we denote $\phi(x):=\frac{1}{4\pi}\kappa_0b(x)+q(x)$. By properties of $b(x)$ and $q(x)$, we have $\phi(x)\in C(\bar{D})\cap C^1(D)$. Suppose $\phi(x)$ attains its maximum at a unique point $\hat{x}\in D$, in the next theorem we claim that the support of the vortex will shrink and get close to $\hat{x}$.
\begin{theorem}\label{thm2}
	Let $(D,b)$ be a continuous regular lake, $f$ be a function satisfying (H1) and (H2). Suppose $\delta(\varepsilon)=(\ln\frac{1}{\varepsilon})^{-1}$ and $\text{max}_{\bar{D}}\phi=\phi(\hat{x})$. Moreover, assume $\phi$ attains its maximum at only one point $\hat{x}\in D$. Then for all sufficiently small $\varepsilon>0$, there exists a family of solutions $(\psi^\varepsilon,\zeta^\varepsilon)$ with the following properties:
	\begin{itemize}
		\item[(i)]For any $p>1$, $\zeta^\varepsilon\in L^p(D)$, $\psi^\varepsilon\in W^{2,p}_{\text{loc}}(D)$ and satisfies
		\begin{equation*}	
		\mathcal{L}\psi^\varepsilon=\zeta^\varepsilon\ \ \text{a.e.} \ \text{in} \ D.
		\end{equation*}
		$(\psi^\varepsilon,\zeta^\varepsilon)$ is of the form
		\begin{equation*}
		\psi^\varepsilon=\mathcal{K}\zeta^\varepsilon+q-\mu^\varepsilon,\ \  \zeta^\varepsilon=\frac{\delta(\varepsilon)}{\varepsilon^2}f(\psi^\varepsilon),\ \
		\int_D\zeta^\varepsilon d\nu=\kappa_0\delta(\varepsilon)
		\end{equation*}
		for some $\mu^\varepsilon$ depending on $\varepsilon$. 
		\item[(ii)] 
		For every $\beta\in(0,1)$, there holds
		\begin{equation}
		diam\left(supp(\zeta^\varepsilon)\right)\le 2\varepsilon^{1-\beta}.
		\end{equation} 
		Furthermore, Let the center of vorticity be
		\begin{equation*}
		X^\varepsilon=\frac{\int_D x\zeta^\varepsilon(x)dm(x)}{\int_D \zeta^\varepsilon(x)dm(x)}.
		\end{equation*}
	   One has $X^\varepsilon\to\hat{x}$ as $\varepsilon\to 0^+$.
	   \item[(iii)] 
	   One has 
	   \begin{equation*}
	   \mathcal{K}\zeta^\varepsilon=\frac{1}{2\pi}\kappa_0b(\hat{x})+o_\varepsilon(1)
	   \end{equation*}
	   and
	   \begin{equation*}
	   \mu^\varepsilon=\phi(\hat{x})+o_\varepsilon(1)=\frac{1}{2\pi}\kappa_0b(\hat{x})+q(\hat{x})+o_\varepsilon(1).
	   \end{equation*}
	   \item[(iv)]
	   Define the rescaled version of $\zeta^\varepsilon$ to be
	   \begin{equation*}
	   \xi^\varepsilon(x)=\frac{\varepsilon^2}{\delta(\varepsilon)}\zeta^\varepsilon(X^\varepsilon+\varepsilon x)=\varepsilon^2\cdot\ln\frac{1}{\varepsilon}\cdot\zeta^\varepsilon(X^\varepsilon+\varepsilon x).
	   \end{equation*}
	   Then every accumulation points of $\xi^\varepsilon(x)$ as $\varepsilon \to 0^+$, in the weak topology of $L^2$, are radially nonincreasing functions.
	\end{itemize}	
\end{theorem}

When the vanishment happens to be sharper than the critical value $1/\ln\frac{1}{\varepsilon}$, the precies asymptotic location of $supp(\zeta^\varepsilon)$ will be the maximum point of $q$. Significantly, we do not require $f$ satisfying $(\text{H2})$ and the support of $\zeta^\varepsilon$ may not concentrate near one point. However, if $q$ has only one maximum point in $\bar{D}$, $supp(\zeta^\varepsilon)$ must concentrate near it. Theorem \ref{thm3} generalizes the result in \cite{CW3}. In order to clarify it, we denote $\mathcal{M}:=\{x\in\bar{D}\ |\ q(x)=\text{max}_{\bar{D}} q\}$.
\begin{theorem}\label{thm3}
Let $(D,b)$ be a continuous regular lake, $f$ be a function satisfying (H1). As $\varepsilon$ goes to zero, $\delta(\varepsilon)$ satisfies $\delta(\varepsilon)\ln\frac{1}{\varepsilon}=o_\varepsilon(1)$. Then for all sufficiently small $\varepsilon>0$, there exists a family of solutions $(\psi^\varepsilon,\zeta^\varepsilon)$ with the following properties:
\begin{itemize}
	\item[(i)]
	For any $p>1$, $\zeta^\varepsilon\in L^p(D)$, $\psi^\varepsilon\in W^{2,p}_{\text{loc}}(D)$ and satisfies
	\begin{equation*}	
	\mathcal{L}\psi^\varepsilon=\zeta^\varepsilon\ \ \text{a.e.} \ \text{in} \ D.
	\end{equation*}
	$(\psi^\varepsilon,\zeta^\varepsilon)$ is of the form
	\begin{equation*}
	\psi^\varepsilon=\mathcal{K}\zeta^\varepsilon+q-\mu^\varepsilon,\ \  \zeta^\varepsilon=\frac{\delta(\varepsilon)}{\varepsilon^2}f(\psi^\varepsilon),\ \
	\int_D\zeta^\varepsilon d\nu=\kappa_0\delta(\varepsilon)
	\end{equation*}
	for some $\mu^\varepsilon$ depending on $\varepsilon$. 
	\item[(ii)] 
	One has 
	\begin{equation*}
	\mathcal{K}\zeta^\varepsilon=o_\varepsilon(1), \ \ \ 
	\mu^\varepsilon=\text{max}_{\bar{D}} q+o_\varepsilon(1).
	\end{equation*}
	Furthermore, as $\varepsilon$ goes to zero, $\text{supp}(\zeta^\varepsilon)$ will shrink to $\mathcal{M}$. Namely, for every $l>0$, there holds $$supp{\zeta^\varepsilon}\subset \mathcal{M}_l:=\{x\in D|dist(x,\mathcal{M})<l\}$$ provided $\varepsilon>0$ is sufficient small. 
\end{itemize}	
\end{theorem}
This paper is organized as follows. In section 2, we studied the variational problem and give the existence and form of solutions. In section 3, we give the proof of Theorem \ref{thm1}. Since the proof is similar to the case $\delta(\varepsilon)=1$, we refer to \cite{CZZ} for more details. In section 4, we give the proof of Theorem \ref{thm2}, in which $\delta(\varepsilon)$ is of critical value. In Section 5, the proof of Theorem \ref{thm3} is presented, which is a general version of the result in \cite{CW3}.

We also bring to the attention of the reader that there is a similar situation with similar results for three-dimensional axisymmetric incompressible inviscid flows where $b=r$, see \cite{AS} \cite{BF1} with stream-function methods and \cite{CWZ} \cite{FT} with vorticity methods. However there is not any classification result (about $\delta(\varepsilon)$) for three-dimensional axisymmetric incompressible inviscid flows yet.

\section{Variational problem}
Denote $f^{-1}$ the inverse function of $f$ in $[0,+\infty)$ and $f^{-1}\equiv0$ in $(-\infty,f(0^+)]$ and denote $F_*(s):=\int_0^sf^{-1}(r)dr$ as the primitive function of $f^{-1}$, one can varifies assumption $(\text{H2})$ is equivalent to 

$(\text{H2}')$ There exists $\vartheta_1\in(0,1)$ such that $$F_*(s)\ge\vartheta_1sf^{-1}(s), \ \ \ \forall s\ge0.$$
Let $\kappa_0>0$ be a fixed number. Define
\begin{equation*}
\mathcal{A}_{\varepsilon,\Lambda}:=\{\zeta\in L^\infty(D)~|~ 0\le \zeta \le \frac{\Lambda\delta(\varepsilon)}{\varepsilon^2}~ \mbox{ a.e. in }D, \int_{D}\zeta(x)d\nu (x)=\kappa_0\delta(\varepsilon) \},
\end{equation*}
It is obvious that if we take $\Lambda>f(0^+)+1$ be a positive number and $\varepsilon$ be sufficiently small, $\mathcal{A}_{\varepsilon,\Lambda}$ is not empty.
Consider the maximization problem of the following functional over $\mathcal{A}_{\varepsilon,\Lambda}$
$$\mathcal{E}(\zeta)=E_q(\zeta)-\mathcal{F}_\varepsilon(\zeta) \ \ \ \zeta\in\mathcal{A}_{\varepsilon,\Lambda},$$
where 
$$E_q(\zeta)=\frac{1}{2}\int_D \zeta(x)\mathcal{K}\zeta(x)d\nu(x)+\int_D q\zeta d\nu(x),$$ 
$$ \mathcal{F}_\varepsilon(\zeta)=\frac{\delta(\varepsilon)}{\varepsilon^2}\int_D F_*(\frac{\varepsilon^2}{\delta(\varepsilon)}\zeta(x))d\nu(x).$$
By our assumptions on $F_*$ and $f^{-1}$, $\mathcal{F}_\varepsilon(\zeta)$ is a convex functional over $\mathcal{A}_{\varepsilon,\Lambda}$.
The following Lemma shows an absolute maximum for $\mathcal{E}$ over $\mathcal{A}_{\varepsilon,\Lambda}$  can be easily found.

\begin{lemma}\label{lem2-1}
	$\mathcal{E}$ is bounded from above and attains its maximum value over $\mathcal{A}_{\varepsilon,\Lambda}$.
\end{lemma}

\begin{proof}
	Since $G(x,y)\in L^1(D\times D)$, $q\in C^2(D)\cap C^1(\bar{D})$, and $R(x,y)\in L^\infty(D\times D)$, we have $$E_q(\zeta)\leq \frac{(\Lambda\delta(\varepsilon))^2(\text{max}_{\bar{D}} b)^3}{2\varepsilon^4}||G||_{L^1(D\times D)}+\kappa_0\delta(\varepsilon)\text{max}_{\bar{D}}q+(\kappa_0\delta(\varepsilon))^2||R||_{L^\infty(D\times D)},\,\,\forall\,\zeta\in \mathcal{A}_{\varepsilon,\Lambda}.$$
	For $\mathcal{F}_\varepsilon$ , by H\"older's inequality we have \[|\mathcal{F}_\varepsilon|\leq \frac{1}{\varepsilon^2}F(\Lambda)|D|,\ \forall\,\zeta\in \mathcal{A}_{\varepsilon,\Lambda}.\]
	Therefore $\mathcal{E}$ is bounded from above over $\mathcal{A}_{\varepsilon,\Lambda}$.
	Let $\{\zeta_{j}\}\subset \mathcal{A}_{\varepsilon,\Lambda}$ be a sequence such that as $j\to +\infty$ 
	$$\mathcal{E}(\zeta_{j}) \to \sup_{\zeta\in \mathcal{A}_{\varepsilon,\Lambda}}\mathcal{E}({\zeta}).$$
    Thus we may assume, up to a subsequence, that $\zeta_j\to\zeta_0$ weakly star in $L^\infty(D)$ as $j\to\infty$ for some $\zeta_0\in \mathcal{A}_{\varepsilon,\Lambda}$.
	It suffices to prove
	\[\mathcal{E}(\zeta_0)\geq\limsup_{j\to\infty}\mathcal{E}(\zeta_j).\]
	Since $G(x,y)\in L^{1}(D\times D)$ and $q\in L^1(D)$, we have
	\begin{equation}\label{2-1}
	\lim_{j\to\infty}E_q(\zeta_j)=E_q(\zeta_0).
	\end{equation}
	According to $(\text{H1})$, $\mathcal{F}_\varepsilon$ is a continuous convex functional. Thus it is also weakly lower semicontinuous and one has
	\begin{equation}\label{2-2}
	\liminf_{j\to +\infty} \mathcal{F}_\varepsilon(\zeta_j)\ge \mathcal{F}_\varepsilon(\zeta_0).
	\end{equation}
	Combining \eqref{2-1} and \eqref{2-2} we get the desired result.
	
\end{proof}

The next lemma gives the profile of maximizer of $\mathcal{E}$ over $\mathcal{A}_{\varepsilon,\Lambda}$.

\begin{lemma}\label{lem2-2}
	Let $\zeta^{\varepsilon,\Lambda}$ be a maximizer of $\mathcal{E}$ over $\mathcal{A}_{\varepsilon,\Lambda}$. Then there exists some $\mu^{\varepsilon,\Lambda}$ such that
	\begin{equation}\label{2-3}
	\zeta^{\varepsilon,\Lambda}=\frac{\delta(\varepsilon)}{\varepsilon^2}f(\psi^{\varepsilon,\Lambda}){\chi}_{\{x\in D \mid 0<\psi^{\varepsilon,\Lambda}(x)<f^{-1}(\Lambda)\}}+\frac{\Lambda\delta(\varepsilon)}{\varepsilon^2}\chi_{_{\{x\in D\mid\psi^{\varepsilon,\Lambda}(x) \geq f^{-1}(\Lambda)\}}} \ \ \mbox{ a.e. in }  D,
	\end{equation}
	where
	\begin{equation}\label{2-4}
	\psi^{\varepsilon,\Lambda}:=\mathcal{K}\zeta^{\varepsilon,\Lambda}+q-\mu^{\varepsilon,\Lambda}.
	\end{equation}
	Moreover, when $\varepsilon$ is sufficiently small, $\mu^{\varepsilon,\Lambda}$ has the following lower bound
	\begin{equation}\label{2-5}
	\mu^{\varepsilon,\Lambda} \ge -f^{-1}(f(0^+)+1)+\text{min}_{\bar{D}}q-1.
	\end{equation}
\end{lemma}

\begin{proof}
	We take a family of test functions as follows
	\begin{equation*}
	\zeta_{s}=\zeta^{\varepsilon,\Lambda}+s({\zeta}-\zeta^{\varepsilon,\Lambda}),\ \ \ s\in[0,1],
	\end{equation*}
	where ${\zeta}$ is an arbitrary element of $\mathcal{A}_{\varepsilon,\Lambda}$. Since $\zeta^{\varepsilon,\Lambda}$ is a maximizer, we have
	\begin{equation*}
	0\ge\frac{d\mathcal E(\zeta_{s})}{ds}\bigg|_{s=0^+}
	=\int_{D}({\zeta}-\zeta^{\varepsilon,\Lambda})\big(\mathcal{K}\zeta^{\varepsilon,\Lambda}+q-f^{-1}(\frac{\delta(\varepsilon)}{\varepsilon^2}\zeta^{\varepsilon,\Lambda}))d\nu(x),
	\end{equation*}
	
	that is,
	\begin{equation*}
	\int_{D}\zeta^{\varepsilon,\Lambda}\big(\mathcal{K}\zeta^{\varepsilon,\Lambda}+q-f^{-1}(\frac{\delta(\varepsilon)}{\varepsilon^2}\zeta^{\varepsilon,\Lambda}))d\nu(x)\ge \int_{D}{\zeta}\big(\mathcal{K}\zeta^{\varepsilon,\Lambda}+q-f^{-1}(\frac{\delta(\varepsilon)}{\varepsilon^2}\zeta^{\varepsilon,\Lambda}))d\nu(x).
	\end{equation*}
	for all ${\zeta}\in \mathcal{A}_{\varepsilon,\Lambda}.$
	By an adaptation of the bathtub principle (see Lieb--Loss \cite{LL}, \S 1.14) we obtain
	\begin{equation}\label{2-6}
	\begin{split}
	\mathcal{K}\zeta^{\varepsilon,\Lambda}+q-\mu^{\varepsilon,\Lambda} &\ge f^{-1}(\frac{\varepsilon^2}{\delta(\varepsilon)}\zeta^{\varepsilon,\Lambda}) \ \ \  \mbox{whenever}\  \zeta^{\varepsilon,\Lambda}=\frac{\Lambda\delta(\varepsilon)}{\varepsilon^2}, \\
	\mathcal{K}\zeta^{\varepsilon,\Lambda}+q-\mu^{\varepsilon,\Lambda} &= f^{-1}(\frac{\varepsilon^2}{\delta(\varepsilon)}\zeta^{\varepsilon,\Lambda})\ \ \  \mbox{whenever}\  0<\zeta^{\varepsilon,\Lambda}<\frac{\Lambda\delta(\varepsilon)}{\varepsilon^2}, \\
	\mathcal{K}\zeta^{\varepsilon,\Lambda}+q-\mu^{\varepsilon,\Lambda} &\le f^{-1}(\frac{\varepsilon^2}{\delta(\varepsilon)}\zeta^{\varepsilon,\Lambda}) \ \ \  \mbox{whenever}\  \zeta^{\varepsilon,\Lambda}=0,
	\end{split}
	\end{equation}
	where $\mu^{\varepsilon,\Lambda}$ is a real number determined by
	$$\mu^{\varepsilon,\Lambda}=\inf\{s\in\mathbb R\mid|\{x\in D\mid\mathcal{K}\zeta^{\varepsilon,\Lambda}+q-f^{-1}(\frac{\varepsilon^2}{\delta(\varepsilon)}\zeta^{\varepsilon,\Lambda})>s\}|\le \frac{\kappa_0\varepsilon^2}{\Lambda}\}.$$
	Now the desired form $\eqref{2-3}$ follows immediately.
	
	Next we prove \eqref{2-5}. Suppose not, as $\varepsilon\to 0^+$ for $x\in D$ there holds
	\begin{equation*}
	\begin{split}
	\psi^{\varepsilon,\Lambda}(x)&=\mathcal{K}\zeta^{\varepsilon,\Lambda}(x)+q(x)-\mu^{\varepsilon,\Lambda}\\
	&=b(x)\int_D G(x,y)\zeta^{\varepsilon,\Lambda}d\nu(y)+\int_D R(x,y)\zeta^{\varepsilon,\Lambda}d\nu(y)+q(x)-\mu^{\varepsilon,\Lambda}\\
	&\ge(\int_D R(x,y)\zeta^{\varepsilon,\Lambda}d\nu(y)+1)+(q-\text{min}_{\bar{D}}q)+f^{-1}(f(0^+)+1)\\
	&\ge f^{-1}(f(0^+)+1),	
	\end{split}
	\end{equation*}
	which implies $\zeta^{\varepsilon,\Lambda}\ge\frac{(f(0^+)+1)\delta(\varepsilon)}{\varepsilon^2}\chi_D$. Thus $\int_{D}\zeta(x)d\nu (x)=\frac{(f(0^+)+1)\delta(\varepsilon)}{\varepsilon^2}|D|>\kappa_0\delta(\varepsilon)$ when $\varepsilon$ is sufficiently small, which leads to a contradiction. Thus the proof is completed.
\end{proof}	

The following Lemma is a variant of the result in \cite{B}.
\begin{lemma}\label{lem2-3}
	Let $(\psi,\zeta)\in W^{2,p}_{loc}(D)\times L^\infty(D)$ for some $p>1$ satisfying $\mathcal{L}\psi=\zeta\ \ a.e \ \ in\ D$. Suppose that $\zeta=f(\psi)\ \ a.e\ \ in\ D$, for some monotonic function $f$. Then $div(\zeta\nabla^{\perp}\psi)=0$ as a distribution. 
\end{lemma}
\begin{proof}
	The proof is the same as Lemma 3.4 of \cite{CWZ} and we omit it.
\end{proof}

Then we turn to analyze the limiting behavior of $\zeta^{\varepsilon,\Lambda}$ with different $\delta(\varepsilon)$ as $\varepsilon$ goes to zero.

\section{Proof of Theorem \ref{thm1}}
In Theorem \ref{thm1}, we denote $\mathcal{S}=\{x\in\bar{D}\ |\ b(x)=\text{max}_{\bar{D}} b\}$ and consider the vanishing rate $\delta(\varepsilon)$ with following properties
\begin{itemize}
	\item[(a)] If $\mathcal{S}\cap D\not=\varnothing$, there holds $\delta(\varepsilon)=o_\varepsilon(1)$ and $(\ln\frac{1}{\varepsilon})^{-1}/\delta(\varepsilon)=o_\varepsilon(1)$.
	\item[(b)] If $\mathcal{S}\subset \partial D$, there holds $\delta(\varepsilon)=o_\varepsilon(1)$ and $(\ln\ln\frac{1}{\varepsilon}/\ln\frac{1}{\varepsilon})/\delta(\varepsilon)=o_\varepsilon(1)$.
\end{itemize}
To simplify our proof, we denote $$H(x,y):=\frac{1}{2\pi}\ln\frac{diam(D)}{|x-y|}-G(x,y).$$
First, we give following estimate for $H$, which is required in the further analysis (see \cite{D2}). 
\begin{lemma}\label{lem3-1}
	For all $x,y\in D$, we have
	\begin{equation*}
	\begin{split}
	\frac{1}{2\pi}\ln\frac{diam(D)}{\text{max}\{|x-y|,dist(x,\partial D),dist(y,\partial D)\}}&\ge H(x,y)\\
	&\ge\frac{1}{2\pi}\ln\frac{diam(D)}{|x+y|+2\text{max}\{dist(x,\partial D),dist(y,\partial D)\}}.
	\end{split}
	\end{equation*}
\end{lemma}

By choosing appropriate test functons, one obtains a lower bound of $\mathcal{E}(\zeta^{\varepsilon,\Lambda})$ in both cases (a) and (b).
\begin{lemma}\label{lem3-2}
		If $\mathcal{S}\cap D\not=\varnothing$, there holds $$\mathcal{E}(\zeta^{\varepsilon,\Lambda})\ge \frac{\kappa_0^2\text{max}_{\bar{D}} b}{4\pi}\delta^2(\varepsilon)\ln{\frac{1}{\varepsilon}}+C\delta(\varepsilon);$$ 
		if $\mathcal{S}\subset \partial D$, there holds $$\mathcal{E}(\zeta^{\varepsilon,\Lambda})\ge \frac{\kappa_0^2\text{max}_{\bar{D}} b}{4\pi}\delta^2(\varepsilon)\ln{\frac{1}{\varepsilon}}-\frac{\alpha\kappa_0^2\text{max}_{\bar{D}} b}{4\pi}\delta^2(\varepsilon)\ln\ln{\frac{1}{\varepsilon}}+C\delta(\varepsilon),$$ 
		where $C$ are constants independent of $\varepsilon$ and $\Lambda$.
\end{lemma}

\begin{proof}
	If $\mathcal{S}\cap D\not=\varnothing$, we fix $\bar{x}\in\mathcal{S}\cap D$ and $0<r<dist(\bar{x},\partial D)$. Let $b_0=\text{inf}_{B_r(\bar{x})}b$ and $$\tilde{\zeta}^{\varepsilon,\Lambda}=\frac{\delta(\varepsilon)b_0}{\varepsilon^2b(x)} \chi_{_{B_{\varepsilon \sqrt{\kappa_0/\pi b_0}}(\bar{x})}}.$$
    it is obvious that $\tilde{\zeta}^{\varepsilon,\Lambda} \in\mathcal{A}_{\varepsilon,\Lambda}$ for all sufficiently small $\varepsilon$. Since $\zeta^{\varepsilon,\Lambda}$ is a maximizer, we have $\mathcal{E}(\zeta^{\varepsilon,\Lambda})\ge \mathcal{E}(\tilde{\zeta}^{\varepsilon,\Lambda})$. A simple calculation yields that
	\begin{equation*}
	\begin{split}
	\mathcal{E}(\zeta^{\varepsilon,\Lambda})&\ge\frac{1}{2}\int_D \tilde{\zeta}^{\varepsilon,\Lambda}(x)\mathcal{K}\tilde{\zeta}^{\varepsilon,\Lambda}d\nu(x)+\int_Dq(x)\tilde{\zeta}^{\varepsilon,\Lambda}(x)d\nu(x)-\frac{\delta(\varepsilon)}{\varepsilon^2}\int_D F(\frac{\varepsilon^2}{\delta(\varepsilon)}\tilde{\zeta}^{\varepsilon,\Lambda}(x))d\nu(x) \\
	&\ge\frac{1}{2}\int_D\int_D\frac{b(x)}{2\pi}\ln\frac{1}{|x-y|}\tilde{\zeta}^{\varepsilon,\Lambda}(x)\tilde{\zeta}^{\varepsilon,\Lambda}(y)d\nu(x)d\nu(y)+\delta(\varepsilon)\text{min}_{\bar{D}}q-\frac{\kappa_0b(\bar{x})}{b_0}\delta(\varepsilon)F(1)\\
	&=\frac{\kappa_0^2}{4\pi}(b(\bar{x})-(\varepsilon \sqrt{\kappa_0/\pi b_0})||b||_{C^{1}(D)})\delta^2(\varepsilon)\ln{\frac{1}{\varepsilon}}+C\delta(\varepsilon)\cdot\\
	&=\frac{\kappa_0^2\text{max}_{\bar{D}} b}{4\pi}\delta^2(\varepsilon)\ln{\frac{1}{\varepsilon}}+C\delta(\varepsilon).\\
	\end{split}
	\end{equation*}
	If $\mathcal{S}\subset \partial D$, take $\bar{x}\in S$ fixed. Since D is a Lipschitz domain, it satisfies an interior cone condition. So we can choose a family of points $\{x^\varepsilon\}\subset D$ and $\theta\in(0,1)$ such that $$\frac{\theta}{(\ln\frac{1}{\varepsilon})^{\frac{1}{\alpha}}}\le dist(x^\varepsilon,\partial D)\le dist(x^\varepsilon,\bar{x})=\frac{1}{(\ln\frac{1}{\varepsilon})^{\frac{1}{\alpha}}}.$$
	Take $r>0$ is sufficiently small, such that $b_0:= \text{inf}_{D\cap B_r(\bar{x})}b>0$. Then we define $$\hat{\zeta}^{\varepsilon,\Lambda}=\frac{\delta(\varepsilon)b_0}{\varepsilon^2b(x)} \chi_{_{B_{\varepsilon \sqrt{\kappa_0/\pi b_0}}(x^\varepsilon)}}. $$
	One can see that $\hat{\zeta}^{\varepsilon,\Lambda}\in\mathcal{A}_{\varepsilon,\Lambda}$ when $\varepsilon$ is sufficiently small. Notice that $\mathcal{E}(\zeta^{\varepsilon,\Lambda})\ge \mathcal{E}(\hat{\zeta}^{\varepsilon,\Lambda})$, by Lemma \ref{lem3-1} and the integral kernel representation of $\mathcal{K}$, we obtain
	\begin{equation*}
	\begin{split}
	\mathcal{E}(\zeta^{\varepsilon,\Lambda})&\ge\frac{1}{2}\int_D \hat{\zeta}^{\varepsilon,\Lambda}(x)\mathcal{K}\hat{\zeta}^{\varepsilon,\Lambda}d\nu(x)+\int_Dq(x)\hat{\zeta}^{\varepsilon,\Lambda}(x)d\nu(x)-\frac{\delta(\varepsilon)}{\varepsilon^2}\int_D F(\frac{\varepsilon^2}{\delta(\varepsilon)}\hat{\zeta}^{\varepsilon,\Lambda}(x))d\nu(x) \\
	&\ge\frac{\kappa_0^2}{4\pi}(b(\bar{x})-(\varepsilon \sqrt{\kappa_0/\pi b_0})^\alpha||b||_{C^{\alpha}(\bar{D})})\delta^2(\varepsilon)\ln{\frac{1}{\varepsilon}}-\frac{\alpha\kappa_0^2\text{sup}_Db}{4\pi}\delta^2(\varepsilon)\ln\ln{\frac{1}{\varepsilon}}+C\delta(\varepsilon)\\
	&=\frac{\kappa_0^2\text{max}_{\bar{D}} b}{4\pi}\delta^2(\varepsilon)\ln{\frac{1}{\varepsilon}}-\frac{\alpha\kappa_0^2\text{max}_{\bar{D}} b}{4\pi}\delta^2(\varepsilon)\ln\ln{\frac{1}{\varepsilon}}+C\delta(\varepsilon).
	\end{split}
	\end{equation*}
	Thus the proof is completed.	
\end{proof}

To study further properties of $\zeta^{\varepsilon,\Lambda}$, we are going to estimate the Lagrange multiplier $\mu^{\varepsilon,\Lambda}$.
\begin{lemma}\label{lem3-3}
	If $\mathcal{S}\cap D\not=\varnothing$, we have
	$$ \mu^{\varepsilon,\Lambda}\ge \frac{\kappa_0\text{max}_{\bar{D}} b}{2\pi}\delta(\varepsilon)\ln{\frac{1}{\varepsilon}}-|1-2\vartheta_1|f^{-1}(\Lambda)+C.$$ 
	If $\mathcal{S}\subset \partial D$, we have
	$$ \mu^{\varepsilon,\Lambda}\ge \frac{\kappa_0\text{max}_{\bar{D}} b}{2\pi}\delta(\varepsilon)\ln{\frac{1}{\varepsilon}}-\frac{\alpha\kappa_0\text{max}_{\bar{D}} b}{2\pi}\delta(\varepsilon)\ln\ln{\frac{1}{\varepsilon}}-|1-2\vartheta_1|f^{-1}(\Lambda)+C.$$ 
	where $\vartheta_1$ is the positive number in $(\text{H2}')$ and $C$ are constants independent of $\varepsilon$ and $\Lambda$.
\end{lemma}
\begin{proof}
	Using \eqref{2-3} and assumption $(\text{H2}')$, we obtain
	\begin{equation}\label{3-1}
	\begin{split}
	2\mathcal E(\zeta^{\varepsilon,\Lambda}) &=  \int_D\zeta^{\varepsilon,\Lambda}\mathcal{K}\zeta^{\varepsilon,\Lambda}d\nu+2\int_Dq\zeta^{\varepsilon,\Lambda}d\nu-\frac{2\delta(\varepsilon)}{\varepsilon^2}\int_DF_*(\frac{\varepsilon^2}{\delta(\varepsilon)}\zeta^{\varepsilon,\Lambda})d\nu \\
	& \le\int_D \zeta^{\varepsilon,\Lambda} \psi^{\varepsilon,\Lambda}d\nu+\int_Dq\zeta^{\varepsilon,\Lambda}d\nu-2\vartheta_1\int_D \zeta^{\varepsilon,\Lambda} f^{-1}(\frac{\varepsilon^2}{\delta(\varepsilon)}\zeta^{\varepsilon,\Lambda})d\nu+\kappa_0\delta(\varepsilon)\mu^{\varepsilon,\Lambda} \\
	&=\int_{0<\zeta<\frac{\delta\Lambda}{\varepsilon^2}}\zeta^{\varepsilon,\Lambda} f^{-1}(\frac{\varepsilon^2}{\delta(\varepsilon)}\zeta^{\varepsilon,\Lambda})d\nu+\int_{\zeta=\frac{\delta\Lambda}{\varepsilon^2}}\zeta^{\varepsilon,\Lambda} \psi^{\varepsilon,\Lambda}d\nu\\
	&\ \ \  -2\vartheta_1\int_D \zeta^{\varepsilon,\Lambda} f^{-1}(\frac{\varepsilon^2}{\delta(\varepsilon)}\zeta^{\varepsilon,\Lambda})d\nu+\int_Dq\zeta^{\varepsilon,\Lambda}d\nu+\kappa_0\delta(\varepsilon)\mu^{\varepsilon,\Lambda} \\
	&\ge\int_{D}\zeta^{\varepsilon,\Lambda} f^{-1}(\frac{\varepsilon^2}{\delta(\varepsilon)}\zeta^{\varepsilon,\Lambda})d\nu-\int_{\zeta=\frac{\delta\Lambda}{\varepsilon^2}}\zeta^{\varepsilon,\Lambda}f^{-1}(\Lambda)d\nu+\kappa_0\delta(\varepsilon)\text{max}_{\bar{D}}q\\
	&\ \ \ +\int_{\zeta=\frac{\delta\Lambda}{\varepsilon^2}}\zeta^{\varepsilon,\Lambda} \psi^{\varepsilon,\Lambda}d\nu-2\vartheta_1\int_D \zeta^{\varepsilon,\Lambda} f^{-1}(\frac{\varepsilon^2}{\delta(\varepsilon)}\zeta^{\varepsilon,\Lambda})d\nu+\kappa_0\delta(\varepsilon)\mu^{\varepsilon,\Lambda}\\
	& \le |1-2\vartheta_1|\kappa_0\delta(\varepsilon) f^{-1}(\Lambda)+\int_D\zeta^{\varepsilon,\Lambda}\left(\psi^{\varepsilon,\Lambda}-\text{max}_{\bar{D}}q-f^{-1}(\Lambda)+\text{min}_{\bar{D}}q-2\right)_+d\nu\\
	&\ \ \ +\kappa_0\delta(\varepsilon)(2\text{max}_{\bar{D}}q-\text{min}_{\bar{D}}q+2)+\kappa_0\delta(\varepsilon)\mu^{\varepsilon,\Lambda}.\\
	\end{split}
	\end{equation}
	Denote $U^{\varepsilon,\Lambda}:=\left(\psi^{\varepsilon,\Lambda}-\text{max}_{\bar{D}}q-f^{-1}(\Lambda)+\text{max}_{\bar{D}}q-2\right)_+$. Since $\mu^{\varepsilon,\Lambda}\ge-f^{-1}(\Lambda)+\text{min}_{\bar{D}}q-1$, by Lemma \ref{lem2-2} we have $U^{\varepsilon,\Lambda}\in H(D)$.
	Hence by integration by parts we have
	\begin{equation}\label{3-2}
	\int_D\frac{|\nabla U^{\varepsilon,\Lambda}|^2}{b^2}d\nu= \int_D \zeta^{\varepsilon,\Lambda}U^{\varepsilon,\Lambda}d\nu.
	\end{equation}
	Then by H\"older's inequality and Sobolev's inequality
	\begin{equation}\label{3-3}
	\begin{split}
	\int_D \zeta^{\varepsilon,\Lambda}U^{\varepsilon,\Lambda}d\nu(x)
	& \le \frac{\delta(\varepsilon)\Lambda}{\varepsilon^2}|\{x\in D\mid\zeta^{\varepsilon,\Lambda}=\frac{\delta(\varepsilon)\Lambda}{\varepsilon^2}\}|^{\frac{1}{2}}\left(\int_D |U^{\varepsilon,\Lambda}|^2dm\right)^{\frac{1}{2}}\\
	& \le \frac{C\delta(\varepsilon)\Lambda }{\varepsilon^2}|\{x\in D\mid\zeta^{\varepsilon,\Lambda}=\frac{\delta(\varepsilon)\Lambda}{\varepsilon^2}\}|^{\frac{1}{2}}\left(\int_D |\nabla U^{\varepsilon,\Lambda}|dm \right)\\
	& \le \frac{C\delta(\varepsilon)\Lambda}{\varepsilon^2}|\{x\in D\mid\zeta^{\varepsilon,\Lambda}=\frac{\delta(\varepsilon)\Lambda}{\varepsilon^2}\}|^{\frac{1}{2}}\left(\int_D \frac{|\nabla U^{\varepsilon,\Lambda}|^2}{b}d\nu \right)\\
	& \le C\kappa_0\delta(\varepsilon)\left(\int_D {\frac{|\nabla U^{\varepsilon,\Lambda}|^2}{b^2}d\nu}\right)^{\frac{1}{2}}.\\
	\end{split}
	\end{equation}
	Here $C$ is a constant independent of $\varepsilon$ and $\Lambda$. From \eqref{3-2} and \eqref{3-3} we conclude that $\int_D \zeta^{\varepsilon,\Lambda}U^{\varepsilon,\Lambda}d\nu(x)=C\delta^2(\varepsilon)$ , which together with \eqref{3-1} and Lemma \ref{lem3-2} leads to the desired result.
\end{proof}

Now we can estimate the diameter of $supp(\zeta^{\varepsilon,\Lambda})$ as $\varepsilon$ tends to zero. To this end, we follows the strategy given in \cite{T} or the method developed in \cite{CWZ}. 
\begin{lemma}\label{lem3-4}
	Suppose $\delta(\varepsilon)$ satisfies restriction $(\text{a})$ or $(\text{b})$ given in Theorem \ref{thm1}, then for every $\gamma\in(0,1)$, There exists some $R_0>1$ such that
	\begin{equation}\label{3-4}
	diam\left(supp(\zeta^{\varepsilon,\Lambda})\right)\le R_0\varepsilon^\gamma,
	\end{equation}
	where $R_0$ may depend on $\Lambda$, but not on $\varepsilon$. We have
	\begin{equation}\label{3-5}
	\lim\limits_{\varepsilon\to0^+}\frac{\ln diam(supp(\zeta^{\varepsilon,\Lambda}))}{\ln\varepsilon}=1.
	\end{equation}
\end{lemma}
\begin{proof}
	For each $x\in supp(\zeta^{\varepsilon,\Lambda})$ there holds
	\begin{equation}\label{3-6}
	\mu^{\varepsilon,\Lambda}\le\mathcal{K}\zeta^{\varepsilon,\Lambda}(x)+q(x)\le\frac{\text{max}_{\bar{D}}b}{2\pi}\int_D \ln\frac{1}{|x-y|}\zeta^{\varepsilon,\Lambda}(y)d\nu(y)+\text{max}_{\bar{D}}q+C,
	\end{equation}
	where $C$ is a constant independent of $\varepsilon$ and $\Lambda$. To normalize the $L_1$ norm of $\zeta^{\varepsilon,\Lambda}$, we denote $\Gamma(y)=\kappa_0^{-1}\delta^{-1}\zeta^{\varepsilon,\Lambda}(y)b(y)$. One can obtain $\int_D\Gamma(y)dm(y)=1$. We devide our proof into three steps.
	
	The first step is to consider the case $\mathcal{S}\cap D\not=\varnothing$. According to Lemma \ref{lem3-3} and \eqref{3-6}, one obtains
	\begin{equation}\label{3-7}
    -\delta^{-1}(\varepsilon)(h(\Lambda)+C)\le\int_D\ln\frac{\varepsilon}{|x-y|}\Gamma(y)dm(y),
	\end{equation}
	where $h(\Lambda)=\frac{2\pi}{\kappa_0\text{max}_{\bar{D}}b}|1-2\vartheta_1|f^{-1}(\Lambda)$. If we take arbitrary $R>1$, a simple calculation yields
	\begin{equation*}
	\begin{split}
	\int_D\ln\frac{\varepsilon}{|x-y|}\Gamma(y)dm(y)&=\int_{D\setminus B_{R\varepsilon}(x)}\ln\frac{\varepsilon}{|x-y|}\Gamma(y)dm(y)+\int_{B_{R\varepsilon}(x)}\ln\frac{\varepsilon}{|x-y|}\Gamma(y)dm(y)\\
	&\le \int_{D\setminus B_{R\varepsilon}(x)}\ln(\frac{1}{R})\Gamma(y)dm(y)+\frac{\Lambda\text{sup}_Db}{\varepsilon^2\kappa_0}\int_{|y|\le\varepsilon}\ln\frac{\varepsilon}{|y|}dy\\
	&\le\ln\frac{1}{R}\int_{D\setminus B_{R\varepsilon}(x)}\Gamma(y)dm(y)+C,
	\end{split}
	\end{equation*}
	where $C$ is a constant not depend on $\varepsilon$. Using \eqref{3-7} we obtain
	\begin{equation}\label{3-8}
	\int_{D\setminus B_{R\varepsilon}(x)}\Gamma(y)dm(y)\le(h(\Lambda)+C)/(\delta(\varepsilon)\ln R).
	\end{equation}
	We can just let $R=e^{\frac{A}{\delta(\varepsilon)}}$, where $A$ is a constant such that the right side of the inequality is less than $\frac{1}{2}$. Hence we have
	\begin{equation}\label{3-9}
	\int_{D\setminus B_{R\varepsilon}(x)}\Gamma(y)dm(y)<\frac{1}{2}.
	\end{equation}
    Now we claim that $diam\left(supp(\zeta^{\varepsilon,\Lambda})\right)\le 2R\varepsilon$. Otherwise there would exist $x_1,x_2\in supp(\zeta^{\varepsilon,\Lambda})$ with the property $B_{R\varepsilon}(x_1)\cap B_{R\varepsilon}(x_2)=\varnothing$. Thus by \eqref{3-9}, there holds
	\begin{equation*}
	\int_D\Gamma(y)dm(y)\ge\int_{B_{R\varepsilon}(x_1)}\Gamma(y)dm(y)+\int_{B_{R\varepsilon}(x_2)}\Gamma(y)dm(y)>1,
	\end{equation*}
	which contradicts with $\int_D\Gamma(y)dm(y)=1$. Recall that $R=\exp(\frac{A}{\delta(\varepsilon)})$. By restriction $(\text{a})$ in Theorem \ref{thm1}, we deduce that for every $\gamma\in(0,1)$ as $\varepsilon$ sufficiently small, there exists some $R_0>1$ not depend on $\varepsilon$ such that $$diam\left(supp(\zeta^{\varepsilon,\Lambda})\right)\le 2R\varepsilon\le R_0\varepsilon^\gamma.$$ 
	
	The second step is to deal with the case  $\mathcal{S}\subset \partial D$. By \eqref{3-6} there holds
	\begin{equation}\label{3-10}
	\delta^{-1}(\varepsilon)(\alpha\ln\ln\frac{1}{\varepsilon}+h(\Lambda)+C)\le\int_D\ln\frac{\varepsilon}{|x-y|}\Gamma(y)dm(y).
	\end{equation}
	Hence instead of \eqref{3-8}, we obtain
	\begin{equation}\label{3-11}
	\int_{D\setminus B_{R\varepsilon}(x)}\Gamma(y)dm(y)\le(\alpha\ln\ln\frac{1}{\varepsilon}+h(\Lambda)+C)/(\delta(\varepsilon)\ln R),
	\end{equation}
	where $C$ does not depend on $\varepsilon$. This time we choose $R=e^{\frac{A}{\delta(\varepsilon)}}(\ln\frac{1}{\varepsilon})^{\frac{2\alpha}{\delta(\varepsilon)}}$ to ensure the right side is less than $\frac{1}{2}$. By restriction $(\text{b})$ in Theorem \ref{thm1}, one can obtain the same result just as the case $\mathcal{S}\cap D\not=\varnothing$.
	
	The last step is to prove \eqref{3-5}. By \eqref{3-4}, we have
	\begin{equation}\label{3-12}
	\liminf\limits_{\varepsilon\to0^+}\frac{\ln diam(supp(\zeta^{\varepsilon,\Lambda}))}{\ln\varepsilon}\ge 1.
	\end{equation}
	On the other hand, from the isodiametric inequality  
	\begin{equation*}
	\kappa_0\delta(\varepsilon)=\int_D\zeta^{\varepsilon,\Lambda}(x)d\nu(x)\le\frac{\pi\text{sup}_Db\Lambda\delta(\varepsilon)}{4\varepsilon^2}diam(supp(\zeta^{\varepsilon,\Lambda}))^2,
	\end{equation*}
	which implies
	\begin{equation}\label{3-13}
	\limsup\limits_{\varepsilon\to0^+}\frac{\ln diam(supp(\zeta^{\varepsilon,\Lambda}))}{\ln\varepsilon}\le 1.
	\end{equation}
	Combining \eqref{3-12} and \eqref{3-13}, we get \eqref{3-5}.
\end{proof}

The next lemma shows $supp(\zeta^{\varepsilon,\Lambda})$ can not get too close to $\partial D$.
\begin{lemma}\label{lem3-5}
	If $\mathcal{S}\cap D\not=\varnothing$, there exists some constant $\eta>0$ not depend on $\varepsilon$ such that 
	\begin{equation}\label{3-14}
	dist(supp(\zeta^{\varepsilon,\Lambda}),\partial D)>\eta.
	\end{equation}
	If $\mathcal{S}\subset \partial D$, there exist constants $C_d,\gamma_0>0$ not depend on $\varepsilon$ such that 
	\begin{equation}\label{3-15}
	dist(supp(\zeta^{\varepsilon,\Lambda}),\partial D)\ge\frac{C_d}{(\ln\frac{1}{\varepsilon})^{\gamma_0}}.
	\end{equation}
\end{lemma}
\begin{proof}
	The idea of this lemma is simple: when the vortex concentrates to the boundary, its total energy is strictly less than it to a inner point, due to the regular part of the Green's function. For the proof we refer to Lemma 2.8 and Lemma 3.7 of \cite{CZZ}. 
\end{proof}

To eliminate the patch part in \eqref{2-3}, we give a refined version of Lemma \ref{3-1}. Recall that the center of vorticity is defined by
\begin{equation*}
X^\varepsilon:=\frac{\int_D x\zeta^{\varepsilon,\Lambda}(x)dm(x)}{\int_D \zeta^{\varepsilon,\Lambda}(x)dm(x)},
\end{equation*}
and the neighborhood of $\mathcal{S}$ is defined by
$$\mathcal{S}_l:=\{x\in D|dist(x,\mathcal{S})<l\}.$$
\begin{lemma}\label{lem3-6}
	For all sufficiently small $\varepsilon>0$, we have 
	\begin{equation}\label{3-16}
	\mu^{\varepsilon,\Lambda}\ge \frac{\kappa_0b(X^\varepsilon)}{2\pi}\delta(\varepsilon)\ln{\frac{1}{\varepsilon}}-b(X^\varepsilon)H(X^\varepsilon,X^\varepsilon)\kappa_0\delta(\varepsilon)-|1-2\vartheta_1|f^{-1}(\Lambda)+C,
	\end{equation}
	where $C$ does not depend on $\varepsilon$ and $\Lambda$.
\end{lemma}
\begin{proof}
	Let us fix small $r>0$ such that $b_0=\text{inf}_{\mathcal{S}_r}>0$. Define $$ \bar{\zeta}^{\varepsilon,\Lambda}=\frac{\delta(\varepsilon)b_0}{\varepsilon^2b(x)} \chi_{_{B_{\varepsilon \sqrt{\kappa_0/\pi b_0}}(X^\varepsilon)}}.$$
	By Lemma \ref{lem3-5}, $\bar{\zeta}^{\varepsilon,\Lambda}$ is in $\mathcal{A}^{\varepsilon,\Lambda}$ if $\varepsilon$ is sufficiently small. We apply the interior estimate for harmonic functions, and deduce that for all $x,y\in supp(\zeta^{\varepsilon,\Lambda})$, there holds
	\begin{equation}\label{3-17}
	\begin{split}
	|H(x,y)-&H(X^\varepsilon,X^\varepsilon)|\\
	&\le|H(x,y)-H(X^\varepsilon,y)|+|H(X^\varepsilon,y)-H(X^\varepsilon,X^\varepsilon)|\\
	&\le\frac{C|x-X^\varepsilon|}{dist(supp(\zeta)^{\varepsilon,\Lambda}),\partial D}|\text{sup}_DH(\cdot,y)|+\frac{C|y-X^\varepsilon|}{dist(supp(\zeta)^{\varepsilon,\Lambda}),\partial D}|\text{sup}_DH(X^\varepsilon,\cdot)|\\
	&\le C_H\varepsilon^{\frac{1}{2}}(\ln\frac{1}{\varepsilon})^{\gamma_0}\ln\ln\frac{1}{\varepsilon},
	\end{split}
	\end{equation} 
where $C_H$ may depend on $\Lambda$ but not on $\varepsilon$, $\gamma_0$ is the positive number in Lemma \ref{lem3-5}. Just let $\varepsilon<\varepsilon_1(\Lambda)$ such that $C_H\varepsilon^{\frac{1}{2}}(\ln\frac{1}{\varepsilon})^{\gamma_0}\ln\ln\frac{1}{\varepsilon}<1$,  one can calculate as in the proof of Lemma \ref{lem3-2} to obtain 
\begin{equation}\label{3-18}
\mathcal{E}(\zeta^{\varepsilon,\Lambda})\ge \frac{\kappa_0^2b(X^\varepsilon)}{4\pi}\delta^2(\varepsilon)\ln{\frac{1}{\varepsilon}}-\frac{1}{2}b(X^\varepsilon)H(X^\varepsilon,X^\varepsilon)\kappa_0^2\delta^2(\varepsilon)+C\delta(\varepsilon),
\end{equation}
where $C$ does not depend on $\varepsilon$ and $\Lambda$. \eqref{3-16} follows from \eqref{3-18} by a same argument as in the proof of Lemma \ref{lem3-3}. The proof is thus completed.
\end{proof}

In the next Lemma, we obtains a prior upper bound of $\psi^{\varepsilon,\Lambda}$ with respect to $\Lambda$.
\begin{lemma}\label{lem3-7}
	For $\varepsilon$ sufficiently small, one has 
	$$\psi^{\varepsilon,\Lambda} \le |1-2\delta_0|f^{-1}(\Lambda)+\frac{\text{max}_{\bar{D}} b}{4\pi}\delta(\varepsilon)\ln\Lambda+C,$$
	where $C$ does not depend on $\varepsilon$ and $\Lambda$.
\end{lemma}
\begin{proof}
	For each $x\in supp(\zeta^{\varepsilon,\Lambda})$, by Lemma \ref{lem3-4} and the bathtub principle we have
\begin{equation*}
\begin{split}
\psi^{\varepsilon,\Lambda}&(x)
\le\frac{b(x)}{2\pi}\int_D \ln\frac{1}{|x-y|}\zeta{^{\varepsilon,\Lambda}}(y)d\nu(y)-b(x)\int_D H(x,y)\zeta{^{\varepsilon,\Lambda}}(y)d\nu(y)-\mu^{_{\varepsilon,\Lambda}}+C\\
&\le(b(X^\varepsilon)+L_0\varepsilon^\alpha||b||_{C^{\alpha}(D)})\frac{\delta(\varepsilon)\Lambda}{2\pi\varepsilon^2}\int_{B_{\varepsilon \sqrt{\kappa_0/\Lambda\pi}}(0)}\ln\frac{1}{|y|}dm(y)-b(X^\varepsilon)H(X^\varepsilon,X^\varepsilon)\kappa_0\delta(\varepsilon)-\mu^{_{\varepsilon,\Lambda}}+C\\
& \le\frac{\kappa_0 b(X^\varepsilon)}{2\pi}\delta(\varepsilon)\ln\frac{1}{\varepsilon}+\frac{\text{max}_{\bar{D}} b}{4\pi}\delta(\varepsilon)\ln\Lambda-b(X^\varepsilon)H(X^\varepsilon,X^\varepsilon)\kappa_0\delta(\varepsilon)-\mu^{_{\varepsilon,\Lambda}}+C,\\
\end{split}
\end{equation*}
where $C$ does not depend on $\varepsilon$ and $\Lambda$. Combining this and Lemma \ref{lem3-6}, we take $\varepsilon<\varepsilon_1(\Lambda)$ and obtain 
$$\psi^{\varepsilon,\Lambda} \le |1-2\delta_0|f^{-1}(\Lambda)+\frac{\text{max}_{\bar{D}} b}{4\pi}\delta(\varepsilon)\ln\Lambda+C,$$
where $C$ is a fixed constant. Thus the proof is complete.
\end{proof}

With Lemma \ref{lem3-7} in hand, we can now eliminate the patch part in \eqref{2-3}.
\begin{lemma}\label{lem3-8}
	Take $\Lambda$ sufficiently large(not depending on $\varepsilon$), then for $\varepsilon>0$ sufficiently small we have
	\begin{equation}\label{3-19}
	|\{x\in D\mid\zeta^{\varepsilon,\Lambda}(x)=\frac{\delta(\varepsilon)\Lambda}{\varepsilon^2}\}|=0.
	\end{equation}
	As a consequence, $\zeta^{\varepsilon,\Lambda}$ has the form
	\begin{equation*}
	\zeta^{\varepsilon,\Lambda}=\frac{\delta(\varepsilon)}{\varepsilon^2}f(\psi^{\varepsilon,\Lambda}).
	\end{equation*}
\end{lemma}
\begin{proof}
	Notice that
	\begin{equation}\label{3-20}
	\psi^{\varepsilon,\Lambda}\ge f^{-1}(\Lambda)\ \ \text{on}\ \  \{x\in D\mid\zeta^{\varepsilon,\Lambda}(x)=\frac{\delta(\varepsilon)\Lambda}{\varepsilon^2}\}.
	\end{equation}
	Combining \eqref{3-20} and Lemma \ref{lem3-7}, we conclude that 
	\begin{equation}\label{3-21}
	(1-|1-2\vartheta_0|)f^{-1}(\Lambda)\le \frac{\text{max}_{\bar{D}} b}{4\pi}\delta(\varepsilon)\ln \Lambda+C   \ \ \text{on}\ \  \{x\in D\mid\omega^{\varepsilon,\Lambda}(x)=\frac{\delta(\varepsilon)\Lambda}{\varepsilon^2}\},
	\end{equation}
	where the constant $C$ does not depend on $\varepsilon$ and $\Lambda$. Note that since $\vartheta_0\in(0,1),$ there holds $1-|1-2\vartheta_0|\in(0,1)$. Recalling the assumption $(\text{H1})$ on $f$, We can choose $\Lambda$ sufficiently large such that $$(1-|1-2\vartheta_0|)f^{-1}(\Lambda)\ge C+1,$$
	where $C$ is the same number in \eqref{3-21}. Since $\delta(\varepsilon)\to 0^+$ as $\varepsilon$ goes to zero, there exists $\varepsilon_2(\Lambda)$ such that $\frac{\text{max}_{\bar{D}} b}{4\pi}\delta(\varepsilon)\ln \Lambda<1$ for every $\varepsilon<\varepsilon_2(\Lambda)$. We can take $\varepsilon<\text{min}\{\varepsilon_1(\Lambda),\varepsilon_2(\Lambda)\}$ (recall $\varepsilon_1(\Lambda)$ in Lemma \ref{3-6}) and obtain
	$$|\{x\in D\mid\zeta^{\varepsilon,\Lambda}(x)=\frac{\delta(\varepsilon)\Lambda}{\varepsilon^2}\}|=0.$$
	Hence we complete the proof.
\end{proof}

In the rest of this section, we fix the parameter $\Lambda$ such that \eqref{3-19} holds. To simplify notations, we shall abbreviate $(\mathcal{A}_{\varepsilon,\Lambda},\zeta^{\varepsilon,\Lambda},\psi^{\varepsilon,\Lambda}, \mu^{\varepsilon,\Lambda})$ as $(\mathcal{A}_{\varepsilon},\zeta^{\varepsilon}, \psi^{\varepsilon}, \mu^{\varepsilon})$. We are going to study the location of $supp(\zeta^\varepsilon)$ as $\varepsilon$ goes to zero. 
\begin{lemma}\label{lem3-9}
    As $\varepsilon\to0^+$, if $X^\varepsilon\to X$, then one has $X\in\mathcal{S}$.
\end{lemma}
\begin{proof}
	The proof is by contradiction. By Lemma \ref{lem3-4}, we can assume there exist a $r_0>0$, a point $Z\notin\mathcal{S}$ and a subsequence $\{\varepsilon_i\}_i^\infty$ such that $\varepsilon_i\to 0^+$ as $i\to +\infty$ and $supp(\zeta^{\varepsilon_i})\subset B_{r_0}(Z)\subset D\setminus\mathcal{S}_{r_0}$. A simple calculation yields to 
	\begin{equation*}
	\begin{split}
	\mathcal{E}(\zeta^{\varepsilon_i})&=\frac{1}{4\pi}\int_D\int_D b(x)\ln\frac{1}{|x-y|}\zeta^{\varepsilon_i}(x)\zeta^{\varepsilon_i}(y)d\nu(x)d\nu(y)-\frac{1}{2}\int_D\int_D b(x)H(x,y)\zeta^{\varepsilon_i}(x)\zeta^{\varepsilon_i}(y)d\nu(x)d\nu(y)\\
	&\ \ \ \ \ +\int_Dq(x)\zeta^{\varepsilon_i}(x)d\nu(x)-\frac{\delta(\varepsilon_i)}{{\varepsilon_i}^{2}}\int_D F(\frac{{\varepsilon_i}^{2}}{\delta(\varepsilon_i)}\zeta^{\varepsilon_i}(x))d\nu(x)\\
	&\le\frac{\kappa_0^2\text{sup}_{B_{r_0}(Z)}b}{4\pi}\delta^2(\varepsilon_i)\ln\frac{1}{\varepsilon_i}+\frac{\text{sup}_{B_{r_0}(Z)}b}{4\pi}\int_D\int_D\ln\frac{\varepsilon_i}{|x-y|}\zeta^{\varepsilon_i}(x)\zeta^{\varepsilon_i}(y)d\nu(x)d\nu(y)+C\delta(\varepsilon_i).\\
	\end{split}
	\end{equation*}
    Using again the bathtub principle, we deduce
	\begin{equation*}
	\begin{split}
	\int_D\int_D\ln\frac{\varepsilon_i}{|x-y|}\zeta^{\varepsilon_i}(x)\zeta^{\varepsilon_i}(y)d\nu(x)d\nu(y)&\le\frac{\Lambda^2(\text{max}_{\bar{D}} b)^2}{{\varepsilon_i}^{4}}\delta^2(\varepsilon_i)\int_{B_{R_0 \varepsilon_i}(Z)}\int_{B_{R_0 \varepsilon_i}(Z)}\ln\frac{\varepsilon_i}{|x-y|}dm(x)dm(y)\\
	&=\Lambda^2(\text{max}_{\bar{D}} b)^2\delta^2(\varepsilon_i)\int_{B_{R_0}(0)}\int_{B_{R_0}(0)}\ln\frac{1}{|x-y|}dm(x)dm(y)\\	
	&=C\delta^2(\varepsilon_i),
	\end{split}
	\end{equation*}
	where $R_0$ is the one in Lemma \ref{lem3-4} and $C$ does not depend on $\varepsilon_i$ and $\Lambda$. Taking $i$ sufficiently large, we have
	\begin{equation}\label{3-22}
	\mathcal{E}(\zeta^{\varepsilon_i})\le\frac{\kappa_0^2\text{sup}_{B_{r_0}(Z)}b}{4\pi}\delta^2(\varepsilon_i)\ln\frac{1}{\varepsilon_i}+C\delta(\varepsilon_i).
	\end{equation}
	But by lemma \ref{lem3-2}, there holds
	\begin{equation}\label{3-23}
	\mathcal{E}(\zeta^\varepsilon)\ge \frac{\kappa_0^2\text{max}_{\bar{D}} b}{4\pi}\delta^2(\varepsilon)\ln{\frac{1}{\varepsilon}}-\frac{\alpha\kappa_0^2\text{max}_{\bar{D}} b}{4\pi}\delta^2(\varepsilon)\ln\ln{\frac{1}{\varepsilon}}+C\delta(\varepsilon).
	\end{equation}	
    Recalling restrictions $(\text{a})$ and $(\text{b})$ in Theorem \ref{thm1}, since $B_{r_0}(Z)\subset D\setminus\mathcal{S}_{r_0}$, \eqref{3-22} contradicts \eqref{3-23}. Thus we have finished the proof.
\end{proof}

From the above proofs , it is not hard to obtain the following asymptotic expansions.
\begin{lemma}\label{lem3-10}
	If $\mathcal{S}\cap D\not=\varnothing$, one has
	\begin{equation*}
	\mu^\varepsilon= \frac{\kappa_0 \text{max}_{\bar{D}} b}{2\pi}\delta(\varepsilon)\ln \frac{1}{\varepsilon}+O_\varepsilon(1).
	\end{equation*}
	If $\mathcal{S}\subset \partial D$, one has
	\begin{equation*}
	\mu^\varepsilon= \frac{\kappa_0 \text{max}_{\bar{D}} b}{2\pi}\delta(\varepsilon)\ln \frac{1}{\varepsilon}- O_\varepsilon(1)\cdot\delta(\varepsilon)\cdot\ln\ln\frac{1}{\varepsilon}+O_\varepsilon(1).
	\end{equation*}
\end{lemma}
\begin{proof}
	Just combine Lemma \ref{lem3-2}, Lemma \ref{lem3-5} and Lemma \ref{lem3-9} to get the expansions.
\end{proof}

By now, we can obtain the asymptotic shape of $\zeta^\varepsilon$. Recall that we define the rescaled version of $\zeta^\varepsilon$ to be 
$$\xi^\varepsilon(x)=\frac{\varepsilon^2}{\delta(\varepsilon)}\zeta^\varepsilon(X^\varepsilon+\varepsilon x).$$ 
\begin{lemma}\label{lem3-11}
	Every accumulation points of $\xi^\varepsilon(x)$ as $\varepsilon \to 0^+$, in the weak topology of $L^2$, are radially nonincreasing functions.
\end{lemma}
\begin{proof}
	The proof of Lemma \ref{lem3-11} is contained in \cite{CZZ}, inspired by a result in Burchard-Guo \cite{BG}.
\end{proof}

\section{Proof of Theorem \ref{thm2}}
In Theorem \ref{thm2}, the vanishing rate $\delta(\varepsilon)$ is of critical value $1/\ln\frac{1}{\varepsilon}$ and we assume $\phi(x)=\frac{\kappa_0}{4\pi}b(x)+q(x)$ attains its maximum at a unique point $\hat{x}\in D$. We first give a upper bound of $\mathcal{K}\zeta^{\varepsilon,\Lambda}$, so that we can eliminate the patch part at the beginning. 

\begin{lemma}\label{lem4-1}
    One has 
    \begin{equation}\label{4-1}
    \mathcal{N}\zeta^{\varepsilon,\Lambda}:=\int_D\ln\frac{diam(D)}{|x-y|}\zeta^{\varepsilon,\Lambda}(y)d\nu(y)\le\kappa_0+C\delta(\varepsilon),	
    \end{equation}
    where $C>0$ does not depend on $\varepsilon$. As a result, for $\varepsilon$ sufficiently small there holds
    \begin{equation}\label{4-2}
    \mathcal{K}\zeta^{\varepsilon,\Lambda}\le \frac{\text{max}_{\bar{D}}b}{2\pi}\kappa_0+1. 
    \end{equation}
    Thus if we take $\Lambda$ sufficiently large, the patch part of $\zeta^{\varepsilon,\Lambda}$ vanishes and \eqref {2-3} can be substituted by
    \begin{equation}\label{4-3}
    \zeta^{\varepsilon,\Lambda}=\frac{\delta(\varepsilon)}{\varepsilon^2}f(\psi^{\varepsilon,\Lambda}){\chi}_{\{x\in D \mid \psi^{\varepsilon,\Lambda}(x)>0\}}.
    \end{equation}
\end{lemma}
\begin{proof}
	By Lemma \ref{lem2-1}, there holds 
	$$\zeta^{\varepsilon,\Lambda}\in\mathcal{A}_{\varepsilon,\Lambda}=\{\zeta\in L^\infty(D)~|~ 0\le \zeta \le \frac{\Lambda\delta(\varepsilon)}{\varepsilon^2}~ \mbox{ a.e. in }D, \int_{D}\zeta(x)d\nu (x)=\kappa_0\delta(\varepsilon) \}.$$ 
	To give a upper bound of $\mathcal{N}\zeta^{\varepsilon,\Lambda}$, we can enlarge $\mathcal{A}_{\varepsilon,\Lambda}$ to be $$\mathcal{B}_{\varepsilon,\Lambda}:=\{\zeta\in L^\infty(D)~|~ 0\le \zeta b(x) \le \frac{\Lambda\text{max}_{\bar{D}}b\delta(\varepsilon)}{\varepsilon^2}=\frac{C_B\delta(\varepsilon)}{\varepsilon^2}~ \mbox{ a.e. in }D, \int_{D}\zeta(x)d\nu (x)=\kappa_0\delta(\varepsilon) \}.$$
	It is obvious that $\text{sup}_{\mathcal{B}_{\varepsilon,\Lambda}}\mathcal{N}\zeta\ge\text{sup}_{\mathcal{A}_{\varepsilon,\Lambda}}\mathcal{N}\zeta$. By the bathtub principle \cite{LL}, we obtains
	\begin{equation*}
	\begin{split}
	\mathcal{N}\zeta^{\varepsilon,\Lambda}&\le\text{sup}_{\mathcal{B}_{\varepsilon,\Lambda}}\mathcal{N}\zeta\\
	&\le\frac{C_B\delta(\varepsilon)}{\varepsilon^2}\int_{B_{\varepsilon\sqrt{\kappa_0/C_b\pi}}(x)}\ln\frac{diam(D)}{|x-y|}dm(y)\\
	&\le \frac{C_B\delta(\varepsilon)}{\varepsilon^2}\int_{B_{\varepsilon\sqrt{\kappa_0/C_b\pi}}(0)}\ln\frac{\varepsilon}{|y|}dm(y)+\frac{C_B\delta(\varepsilon)}{\varepsilon^2}\ln\frac{diam(D)}{\varepsilon}\int_{B_{\varepsilon\sqrt{\kappa_0/C_b\pi}}(0)}dm(y)\\
	&\le \frac{1}{2}\kappa_0\delta(\varepsilon)(1+2\ln(diam(D))+\ln\frac{C_B\pi}{\kappa_0})+\kappa_0\delta(\varepsilon)\ln\frac{1}{\varepsilon}\\
	&=\kappa_0+C\delta(\varepsilon),
	\end{split}
	\end{equation*}
	where $C$ is a positive constant not depend on $\varepsilon$. Since $C_B=\Lambda\text{max}_{\bar{D}}b$, we can take $\varepsilon<\varepsilon_3(\Lambda)$ such that $$\frac{\text{max}_{\bar{D}}b}{2\pi}\mathcal{N}\zeta^{\varepsilon,\Lambda}+\kappa_0\delta(\varepsilon)||R||_{L^\infty(D\cdot D)}< \frac{\text{max}_{\bar{D}}b}{2\pi}\kappa_0+1.$$
	By Definition \ref{def2}, we have $\mathcal{K}\zeta^{\varepsilon,\Lambda}\le\frac{b(x)}{2\pi}\mathcal{N}\zeta^{\varepsilon,\Lambda}+\int_DR(x,y)\zeta^{\varepsilon,\Lambda}d\nu(y)$, and hence obtain \eqref{4-2}. If we fix $\Lambda$ sufficiently large such that
	\begin{equation}\label{4-4}
	f^{-1}(\Lambda)>\text{max}_{\bar{D}}\mathcal{K}\zeta^{\varepsilon,\Lambda}+\text{max}_{\bar{D}}q-\text{min}\mu^{\varepsilon,\Lambda},
	\end{equation}
	we obtain
	$$f^{-1}(\Lambda)>\mathcal{K}\zeta^{\varepsilon,\Lambda}(x)+q(x)-\mu^{\varepsilon,\Lambda}\ \ \ \ \forall x\in D.$$
	So it is obvious that 
	\begin{equation}\label{4-5}
	|\{x\in D\mid\zeta^{\varepsilon,\Lambda}(x)=\frac{\delta(\varepsilon)\Lambda}{\varepsilon^2}\}|=0,
	\end{equation}
	which indicates the patch part of $\zeta^{\varepsilon,\Lambda}$ vanishes and \eqref{4-3} holds. Thus the proof is complete.	
\end{proof}

From now on, we fix $\Lambda$ sufficiently large such that \eqref{4-4} holds, so we can abbreviate $(\mathcal{A}_{\varepsilon,\Lambda},\zeta^{\varepsilon,\Lambda},\psi^{\varepsilon,\Lambda}, \mu^{\varepsilon,\Lambda})$ as $(\mathcal{A}_{\varepsilon},\zeta^{\varepsilon}, \psi^{\varepsilon}, \mu^{\varepsilon})$. We are to obtain an upper bound of $E_q(\zeta^\varepsilon)$ in the following Lemma.
\begin{lemma}\label{lem4-2}
	For $\varepsilon$ sufficiently small one has
	\begin{equation}\label{4-6}
	\frac{1}{2}\mathcal{K}\zeta^\varepsilon+q\le\phi(\hat{x})+C\delta(\varepsilon)
	\end{equation}
	and
	\begin{equation}\label{4-7}
	E_q(\zeta^\varepsilon)=\frac{1}{2}\int_D \zeta^\varepsilon\mathcal{K}\zeta^\varepsilon d\nu+\int_D q\zeta^\varepsilon d\nu(x)\le\phi(\hat{x})\kappa_0\delta(\varepsilon)+C\delta^2(\varepsilon)
	\end{equation}
	where $C>0$ does not depend on $\varepsilon$.
\end{lemma}
\begin{proof}
	We have $$E_q(\zeta^\varepsilon)\le\int_D\zeta^\varepsilon d\nu \cdot\text{max}_{\bar{D}}(\frac{1}{2}\mathcal{K}\zeta^\varepsilon+q)=\kappa_0\delta(\varepsilon)\text{max}_{\bar{D}}(\frac{1}{2}\mathcal{K}\zeta^\varepsilon+q).$$
	On the other hand, by \eqref{4-1} and the definition of $\hat{x}$ there holds
	$$\frac{1}{2}\mathcal{K}\zeta^\varepsilon+q\le\frac{\kappa_0}{4\pi}b(x)+q(x)+C\delta(\varepsilon)\le\phi(\hat{x})+C\delta(\varepsilon),$$
	where $C>0$ does not depend on $\varepsilon$. Hence we get an upper bound of $E_q(\zeta^\varepsilon)$.
\end{proof}

As Lemma \ref{lem3-2} in Section 3, we choose a proper test function to get a lower bound of $\mathcal{E}(\zeta^\varepsilon)=E_q(\zeta^\varepsilon)-\mathcal{F}_\varepsilon(\zeta^\varepsilon)$.
\begin{lemma}\label{lem4-3}
	As $\varepsilon\to 0^+$, one has
	\begin{equation}\label{4-8}
	\mathcal{E}(\zeta^\varepsilon)=E_q(\zeta^\varepsilon)-\mathcal{F}_\varepsilon(\zeta^\varepsilon)\ge\phi(\hat{x})\kappa_0\delta(\varepsilon)-o_\varepsilon(1)\cdot\delta(\varepsilon)
	\end{equation}
	and
	\begin{equation}\label{4-9}
	E_q(\zeta^\varepsilon)=\phi(\hat{x})\kappa_0\delta(\varepsilon)+V_E(\varepsilon)\cdot\delta(\varepsilon),
	\end{equation}
	where $V_E(\varepsilon)$ is of order $o_\varepsilon(1)$. Moreover, as $\varepsilon\to0^+$ there holds $$-C\text{max}\{f^{-1}(\delta(\varepsilon)),\frac{\ln\ln\frac{1}{\varepsilon}}{\ln\frac{1}{\varepsilon}}\}\le V_E(\varepsilon)\le C\delta(\varepsilon),$$
    where the constant $C>0$ does not depend on $\varepsilon$.	
\end{lemma}
\begin{proof}
	We choose some $r$ satisfying $0<r<dist(\hat{x},\partial D)$. Let $b_0=\text{inf}_{B_r(\hat{x})}b$ and  $$\check{\zeta}^\varepsilon=\frac{\delta^2(\varepsilon)b_0}{\varepsilon^2b(x)} \chi_{_{B_{\varepsilon \sqrt{\kappa_0/\pi b_0\delta(\varepsilon)}}(\hat{x})}}.$$
	We can check easily $\check{\zeta}^\varepsilon \in\mathcal{A}_\varepsilon$ for all sufficiently small $\varepsilon$. Since $\zeta^\varepsilon$ is a maximizer, we have $\mathcal{E}(\zeta^\varepsilon)\ge \mathcal{E}(\check{\zeta}^\varepsilon)$. A simple calculation yields that
	\begin{equation*}
	\begin{split}
	\frac{1}{2}\int_D\check\zeta^\varepsilon\mathcal{K}\check\zeta^\varepsilon d\nu&\ge\frac{1}{4\pi}(b(\hat{x})+(\frac{C\varepsilon}{\sqrt{\delta(\varepsilon)}})||b||_{C^1(D)})\kappa_0^2\delta^2(\varepsilon)(\ln\frac{1}{\varepsilon}-\frac{1}{2}\ln\ln\frac{1}{\varepsilon})\\
	&=\frac{b(\hat{x})}{4\pi}\kappa_0^2\delta(\varepsilon)-\frac{C\ln\ln\frac{1}{\varepsilon}}{\ln\frac{1}{\varepsilon}}\cdot\delta(\varepsilon)+C\varepsilon\cdot\sqrt{\delta(\varepsilon)},\\
	\int_D\check\zeta^\varepsilon qd\nu&=(q(\hat{x})+(\frac{C\varepsilon}{\sqrt{\delta(\varepsilon)}})||q||_{C^1(\bar{D})})\kappa_0\delta(\varepsilon)\\
	&=q(\hat{x})\kappa_0\delta(\varepsilon)+C\varepsilon\cdot\sqrt{\delta(\varepsilon)},
	\end{split}
	\end{equation*}
	where $C>0$ does not depend on $\varepsilon$. Since $F_*(t)\le tf^{-1}(t)$, we have 
	$$-\mathcal{F}_\varepsilon(\check\zeta^\varepsilon)\ge-\int_D\check\zeta^\varepsilon f^{-1}(\delta(\varepsilon))d\nu\ge-f^{-1}(\delta(\varepsilon))\kappa_0\delta(\varepsilon)=-Cf^{-1}(\delta(\varepsilon))\cdot\delta(\varepsilon).$$
	We obtain \eqref{4-8} by adding all these three terms up. Now we can combine \eqref{4-7} and \eqref{4-8} to obtain
	$$\phi(\hat{x})\kappa_0\delta(\varepsilon)-Cf^{-1}(\delta(\varepsilon))\cdot\delta(\varepsilon)-\frac{C\ln\ln\frac{1}{\varepsilon}}{\ln\frac{1}{\varepsilon}}\cdot\delta(\varepsilon)+C\varepsilon(\ln\frac{1}{\varepsilon})^\frac{1}{2}\cdot\delta(\varepsilon)\le E_q(\zeta^\varepsilon)\le\phi(\hat{x})\kappa_0\delta(\varepsilon)+C\delta^2(\varepsilon).$$
	This gives \eqref{4-9}. Thus the proof is complete.
\end{proof}

Now we are going to estimate $\mathcal{F}_\varepsilon(\zeta^\varepsilon)$. Thanks to assumption $(\text{H2}')$, the energy contributed by the vortex core can be controlled by $\mathcal{F}_\varepsilon(\zeta^\varepsilon)$.
\begin{lemma}\label{lem4-4}
	As $\varepsilon\to 0^+$, one has 
	\begin{equation}\label{4-10}
	\mathcal{F}_\varepsilon(\zeta^\varepsilon)\le o_\varepsilon(1)\cdot\delta(\varepsilon)
	\end{equation}
	and
	\begin{equation}\label{4-11}
	\int_D\zeta^\varepsilon\psi^\varepsilon d\nu=\int_D\zeta^\varepsilon(\mathcal{K}\zeta^\varepsilon+q-\mu^\varepsilon)d\nu=\int_D \zeta^\varepsilon f^{-1}(\frac{\varepsilon^2}{\delta(\varepsilon)}\zeta^\varepsilon)d\nu(x)\le o_\varepsilon(1)\cdot\delta(\varepsilon).
	\end{equation}
\end{lemma}
\begin{proof}
	By \eqref{4-7} and \eqref{4-8}, one obtains
	$$\mathcal{F}_\varepsilon(\zeta^\varepsilon)\le o_\varepsilon(1)\cdot\delta(\varepsilon)+o_\varepsilon(1)\cdot\delta(\varepsilon)=o_\varepsilon(1)\cdot\delta(\varepsilon).$$
	By assumption $(\text{H2}')$, there holds
	$$\exists\vartheta_1\in(0,1) \ \ \ s.t.\ \ \ F_*(s)\ge\vartheta_1sf^{-1}(s), \ \ \ \forall s>0$$
	Hence one can deduce that
	$$\int_D \zeta^\varepsilon f^{-1}(\frac{\varepsilon^2}{\delta(\varepsilon)}\zeta^\varepsilon)d\nu\le \frac{\delta(\varepsilon)}{\vartheta_1\varepsilon^2}\int_D F_*(\frac{\varepsilon^2}{\delta(\varepsilon)}\zeta^\varepsilon)d\nu=\frac{1}{\vartheta_1}\mathcal{F}_\varepsilon(\zeta^\varepsilon)\le o_\varepsilon(1)\cdot\delta(\varepsilon).$$
	Recalling \eqref{4-3}, the patch part of $\zeta^\varepsilon$ vanishes, which implies
	$$\int_D\zeta^\varepsilon\psi^\varepsilon d\nu=\int_D\zeta^\varepsilon(\mathcal{K}\zeta^\varepsilon+q-\mu^\varepsilon)d\nu=\int_D \zeta^\varepsilon f^{-1}(\frac{\varepsilon^2}{\delta(\varepsilon)}\zeta^\varepsilon)d\nu.$$
	Thus we get \eqref{4-11} and complete the proof. 
\end{proof}

The following lemma indicates $\zeta^\varepsilon$ will concentrat near $\hat{x}$ in some extent. Before giving the lemma, we define a neighborhood of $\hat{x}$ to be
$$\mathcal{P}_\varepsilon:=\{x\in D\ |\ \phi(\hat{x})-\phi(x)<V_P(\varepsilon)\},$$
where $V_P(\varepsilon)$ of order $o_\varepsilon(1)$ is given by
$$
V_P(\varepsilon)=\left\{
\begin{array}{lll}
\left(f^{-1}(\delta(\varepsilon))\right)^{\frac{1}{2}}                                                               &    \text{if} & f^{-1}(\delta(\varepsilon))>\frac{\ln\ln\frac{1}{\varepsilon}}{\ln\frac{1}{\varepsilon}}, \\
\left(\frac{\ln\ln\frac{1}{\varepsilon}}{\ln\frac{1}{\varepsilon}}\right)^{\frac{1}{2}}                  &    \text{if} & f^{-1}(\delta(\varepsilon))\le\frac{\ln\ln\frac{1}{\varepsilon}}{\ln\frac{1}{\varepsilon}}.
\end{array}
\right.
$$
Since $\phi(x)\in C(\bar{D})\cap C^1(D)$ and $\hat{x}\in D$ is the only maximum point of $\phi(x)$, one has $diam(\mathcal{P}_\varepsilon)=o_\varepsilon(1)$. Thus we can write $\zeta^\varepsilon$ as 
$$\zeta^\varepsilon=\zeta^\varepsilon_{p_1}+\zeta^\varepsilon_{p_2},$$ where
$\zeta^\varepsilon_{p_1}=\zeta^\varepsilon{\chi}_{\mathcal{P}_\varepsilon}$ and $\zeta^\varepsilon_{p_2}=\zeta^\varepsilon-\zeta^\varepsilon_{p_1}$. We denote $\kappa^\varepsilon_{p_1}=\frac{1}{\delta(\varepsilon)}\int_D\zeta^\varepsilon_{p_1}d\nu$ and $\kappa^\varepsilon_{p_2}=\frac{1}{\delta(\varepsilon)}\int_D\zeta^\varepsilon_{p_2}d\nu$. 
\begin{lemma}\label{lem4-5}
As $\varepsilon\to 0^+$, one has
\begin{equation}\label{4-12}
\kappa^\varepsilon_{p_1}/\kappa_0=1-o_\varepsilon(1) \ \ \ and\ \ \ \kappa^\varepsilon_{p_2}/\kappa_0=o_\varepsilon(1).
\end{equation}		
\end{lemma}
\begin{proof}
	The proof is by contradiction. Suppose there exists a constant $C_P\in(0,1)$, such that $\kappa^\varepsilon_{p_2}\ge C_P\kappa_0$ for a subsequence of $\varepsilon\to0^+$. Using \eqref{4-6}, there holds
	\begin{equation*}
	\begin{split}
	E_q(\zeta^\varepsilon)&=\int_D(\frac{1}{2}\mathcal{K}\zeta^\varepsilon+q)\zeta^\varepsilon_{p_1} d\nu+\int_D(\frac{1}{2}\mathcal{K}\zeta^\varepsilon+q)\zeta^\varepsilon_{p_2} d\nu\\
	&\le(\phi(\hat{x})+C\delta(\varepsilon))\kappa^\varepsilon_{p_1}\delta(\varepsilon)+\text{sup}_{D/\mathcal{P}_\varepsilon}\phi(x)\cdot\kappa^\varepsilon_{p_2}\delta(\varepsilon)\\
	&=\phi(\hat{x})\kappa^\varepsilon_{p_1}\delta(\varepsilon)+(\phi(\hat{x})-V_P(\varepsilon))\kappa^\varepsilon_{p_2}\delta(\varepsilon)+C\delta^2(\varepsilon)\\
	&\le\phi(\hat{x})\kappa_0\delta(\varepsilon)-V_P(\varepsilon)\cdot C_P\kappa_0\delta(\varepsilon)+C\delta^2(\varepsilon).
	\end{split}
	\end{equation*}
	Recalling the definition of $V_P(\varepsilon)$ and $V_E(\varepsilon)$, as $\varepsilon$ sufficiently small we have 
	$$E_q(\zeta^\varepsilon)<\phi(\hat{x})\kappa_0\delta(\varepsilon)+V_E(\varepsilon)\cdot\delta(\varepsilon),$$
	which contradicts \eqref{4-9} since $V_P(\varepsilon)\cdot C_P\kappa_0\delta(\varepsilon)$ dominates the other terms. Hence we obtain \eqref{4-12}. (With more careful analysis, we can prove $\kappa^\varepsilon_{p_2}/\kappa_0\le C(V_P(\varepsilon))^{\frac{1}{2}}$. We omit the proof here.)
\end{proof}

Due to Lemma \ref{lem4-5}, when $\varepsilon$ is sufficiently small, $E_q(\zeta^\varepsilon)$ can be separated into two parts. From this observation, we give an estimate of $\mu^\varepsilon$.
\begin{lemma}\label{lem4-6}
	As $\varepsilon\to 0^+$, one has
	\begin{equation}\label{4-13}
	\int_Dq\zeta^\varepsilon d\nu=q(\hat{x})\cdot\kappa_0\delta(\varepsilon)+o_\varepsilon(1)\cdot\delta(\varepsilon)
	\end{equation}
	and
	\begin{equation}\label{4-14}
	\int_D\zeta^\varepsilon\mathcal{K}\zeta^\varepsilon d\nu=\frac{\kappa_0}{2\pi}b(\hat{x})\cdot\kappa_0\delta(\varepsilon)+o_\varepsilon(1)\cdot\delta(\varepsilon).
	\end{equation}
	Hence there holds
	\begin{equation}\label{4-15}
	\mu^\varepsilon=\frac{\kappa_0}{2\pi}b(\hat{x})+q(\hat{x})+o_\varepsilon(1)
	\end{equation}
\end{lemma}
\begin{proof}
	To prove \eqref{4-13}, we use the regularity of $b(x)$ and $q(x)$. As $\varepsilon\to 0^+$, for all $x\in \mathcal{P}_\varepsilon$ we have
	$$b(x)=b(\hat{x})+o_\varepsilon(1)\ \ \ and \ \ \ q(x)=q(\hat{x})+o_\varepsilon(1).$$   
	One obtains
	\begin{equation*}
	\begin{split}
	\int_Dq\zeta^\varepsilon d\nu&=\int_Dq\zeta^\varepsilon_{p_1} d\nu+\int_Dq\zeta^\varepsilon_{p_2} d\nu\\
	&=(q(\hat{x})+o_\varepsilon(1))(\kappa_0-o_\varepsilon(1)\cdot\kappa_0)\delta(\varepsilon)+o_\varepsilon(1)\cdot\kappa_0\delta(\varepsilon)\\
	&=q(\hat{x})\cdot\kappa_0\delta(\varepsilon)+o_\varepsilon(1)\cdot\delta(\varepsilon).
	\end{split}
	\end{equation*}
	Combining \eqref{4-9} and \eqref{4-13} we have \eqref{4-14}. Using \eqref{4-11}, there holds
	$$\int_D\zeta^\varepsilon\psi^\varepsilon d\nu=\int_D\zeta^\varepsilon(\mathcal{K}\zeta^\varepsilon+q-\mu^\varepsilon)d\nu\le o_\varepsilon(1)\cdot\delta(\varepsilon).$$
	Thus as $\varepsilon$ goes to zero, we have
	$$\mu^\varepsilon\cdot\kappa_0\delta(\varepsilon)=\int_D\zeta^\varepsilon\mathcal{K}\zeta^\varepsilon d\nu+\int_Dq\zeta^\varepsilon d\nu+o_\varepsilon(1)\cdot\delta(\varepsilon).$$
	According to \eqref{4-13} and \eqref{4-14}, \eqref{4-15} follows and we complete the proof.	
\end{proof}
 
For every $l_0$ satisfies $0<l_0<dist(\hat{x},\partial D)$, we define $\zeta^\varepsilon_1:=\zeta^\varepsilon\chi_{B_{l_0}(\hat{x})}$ and $\zeta^\varepsilon_2:=\zeta^\varepsilon-\zeta^\varepsilon_1$. Correspondingly, we denote $\kappa^\varepsilon_1=\frac{1}{\delta(\varepsilon)}\int_D\zeta^\varepsilon_1d\nu$ and $\kappa^\varepsilon_2=\frac{1}{\delta(\varepsilon)}\int_D\zeta^\varepsilon_2d\nu$. By Lemma \ref{lem4-5}, as $\varepsilon\to0^+$, we have 
$$\kappa^\varepsilon_1/\kappa_0=1-o_\varepsilon(1) \ \ \ and\ \ \ \kappa^\varepsilon_2/\kappa_0=o_\varepsilon(1).$$
In the next lemma, we argue that for each $x\in supp(\zeta^\varepsilon_1)$, $\mathcal{K}\zeta^\varepsilon(x)+q(x)$ has a uniform upper bound.
\begin{lemma}\label{lem4-7}
 	As $\varepsilon\to 0^+$, for each $x\in supp(\zeta^\varepsilon_1)$ one has
 	\begin{equation}\label{4-16}
 	\mathcal{K}\zeta^\varepsilon(x)+q(x)\le\frac{b(\hat{x})}{2\pi}\int_D\ln\frac{1}{|x-y|}\zeta^\varepsilon(y)d\nu(y)+q(\hat{x})+C_Ll_0+o_\varepsilon(1),
 	\end{equation}
 	where $C_L>0$ does not depend on $\varepsilon$.
\end{lemma}
\begin{proof}
	For each $x\in supp(\zeta^\varepsilon_1)$, there holds
	$$\mathcal{K}\zeta^\varepsilon(x)+q(x)\le\frac{1}{2\pi}(b(\hat{x})+l_0||b||_{C^1(D)})\int_D\ln\frac{1}{|x-y|}\zeta^\varepsilon(y)d\nu(y)+q(\hat{x})+l_0||q||_{C^1(D)}+o_\varepsilon(1).$$
	By Lemma \ref{lem4-1}, we can take $\varepsilon$ sufficiently small such that 
	$$\int_D\ln\frac{diam(D)}{|x-y|}\zeta^\varepsilon(y)d\nu(y)=\mathcal{N}\zeta^\varepsilon(x)\le\kappa_0+1.$$
	Thus we have
	$$\mathcal{K}\zeta^\varepsilon(x)+q(x)\le\frac{b(\hat{x})}{2\pi}\int_D\ln\frac{1}{|x-y|}\zeta^\varepsilon(y)d\nu(y)+q(\hat{x})+((\kappa_0+1)||b||_{C^1(D)}+||q||_{C^1(D)})l_0+o_\varepsilon(1).$$
	Finally, we denote $C_L=(\kappa_0+1)||b||_{C^1(D)}+||q||_{C^1(D)}$ and complete the proof.	
\end{proof}

Using the strategy in \cite{T}, we are going to prove $diam\left(supp(\zeta^\varepsilon_1)\right)$ is of order $\varepsilon^{1-\beta}$ for every $\beta\in(0,1)$. However, the proof is different from Lemma \ref{lem3-4}.
\begin{lemma}\label{lem4-8}
	For every $\beta\in(0,1)$, there exists $l_0(\beta)>0$ such that for every $l_0<l_0(\beta)$ and $\varepsilon$ sufficiently small, one has
	\begin{equation}\label{4-17}
	diam\left(supp(\zeta^\varepsilon_1)\right)\le 2\varepsilon^{1-\beta},
	\end{equation}	
\end{lemma}
\begin{proof}
	By \eqref{4-15} and \eqref{4-16}, for each $x\in supp(\zeta^\varepsilon_1)$ there holds
	\begin{equation}\label{4-18}
	\mu^{\varepsilon,\Lambda}=\frac{\kappa_0}{2\pi}b(\hat{x})+q(\hat{x})+o_\varepsilon(1)\le\frac{b(\hat{x})}{2\pi}\int_D \ln\frac{1}{|x-y|}\zeta^{\varepsilon,\Lambda}(y)d\nu(y)+q(\hat{x})+C_Ll_0+o_\varepsilon(1).
	\end{equation}
	To normalize the $L_1$ norm of $\zeta^\varepsilon$, we denote 
	$$\Gamma(y)=\kappa_0^{-1}\delta^{-1}\zeta^\varepsilon(y)b(y), \ \ \Gamma_1(y)=\kappa_0^{-1}\delta^{-1}\zeta^\varepsilon_1(y)b(y), \ \ \Gamma_2(y)=\kappa_0^{-1}\delta^{-1}\zeta^\varepsilon_2(y)b(y).$$ 
	One obtains 
	$$\int_D\Gamma(y)dm(y)=1, \ \ \int_D\Gamma_1(y)dm(y)=1-o_\varepsilon(1),\ \ \int_D\Gamma_2(y)dm(y)=o_\varepsilon(1).$$ 
	Notice that $\delta(\varepsilon)=1/\ln\frac{1}{\varepsilon}$. By \eqref{4-18}, for each $x\in supp(\zeta^\varepsilon_1)$ one has
	\begin{equation}\label{4-19}
	-(o_\varepsilon(1)+C_Ll_0)\ln\frac{1}{\varepsilon}\le\int_D\ln\frac{\varepsilon}{|x-y|}\Gamma(y)dm(y),
	\end{equation}
	As in Lemma \ref{lem3-4}, if we choose arbitrary $R>1$, a simple calculation yields 
	\begin{equation*}
	\begin{split}
	\int_D\ln\frac{\varepsilon}{|x-y|}\Gamma(y)dm(y)&=\int_{D\setminus B_{R\varepsilon}(x)}\ln\frac{\varepsilon}{|x-y|}\Gamma(y)dm(y)+\int_{B_{R\varepsilon}(x)}\ln\frac{\varepsilon}{|x-y|}\Gamma(y)dm(y)\\
	&=\ln\frac{1}{R}\int_{D\setminus B_{R\varepsilon}(x)}\Gamma(y)dm(y)+C_\Gamma,
	\end{split}
	\end{equation*}
	where the constant $C_\Gamma>0$ does not depend on $\varepsilon$ and $l_0$. Using \eqref{4-19}, for each $x\in supp(\zeta^\varepsilon_1)$, there holds
	\begin{equation}\label{4-20}
	\int_{D\setminus B_{R\varepsilon}(x)}\Gamma(y)dm(y)\le \left(C_\Gamma+(o_\varepsilon(1)+C_Ll_0)\ln\frac{1}{\varepsilon}\right)/\ln R.
	\end{equation}
	For every $\beta\in(0,1)$, we can take $R=(\frac{1}{\varepsilon})^\beta$ and $l_0<l_0(\beta)$ such that $C_Ll_0<\frac{\beta}{12}$. We also let $\varepsilon$ be sufficiently small such that $\varepsilon<\text{exp}(\frac{-6C_\Gamma}{\beta})$ and the term $o_\varepsilon(1)$ in \eqref{4-20} is less than $\frac{\beta}{12}$. Then for each $x\in supp(\zeta^\varepsilon_1)$, we have
	\begin{equation}\label{4-21}
	\int_{D\setminus B_{R\varepsilon}(x)}\Gamma(y)dm(y)<\left(\frac{\beta}{6}+\frac{\beta}{12}+\frac{\beta}{12}\right)/\beta=\frac{1}{3}.
	\end{equation}
	Now we claim that $diam\left(supp(\zeta^\varepsilon_1)\right)\le 2R\varepsilon$. Otherwise there would exist $x_1,x_2\in supp(\zeta^\varepsilon_1)$ with the property $B_{R\varepsilon}(x_1)\cap B_{R\varepsilon}(x_2)=\varnothing$. We let $\varepsilon<\varepsilon(l_0)$ such that $\int_D\Gamma_2(y)dm(y)<\frac{1}{3}$. By \eqref{4-21}, there holds
	\begin{equation*}
	\begin{split}
	\int_D\Gamma_1(y)dm(y)&\ge\int_{B_{R\varepsilon}(x_1)}\Gamma_1(y)dm(y)+\int_{B_{R\varepsilon}(x_2)}\Gamma_1(y)dm(y)\\
	&\ge\int_{B_{R\varepsilon}(x_1)}\Gamma(y)dm(y)+\int_{B_{R\varepsilon}(x_2)}\Gamma(y)dm(y)-\int_D\Gamma_2(y)dm(y)\\
	&>\frac{2}{3}+\frac{2}{3}-\frac{1}{3}\\
	&>1
	\end{split}
	\end{equation*}
	which contradicts $\int_D\Gamma_1(y)dm(y)\le1$. Recalling $R=(\frac{1}{\varepsilon})^\beta$, we deduce that for every $\beta\in(0,1)$ and $\varepsilon$ sufficiently small, there holds
	$$diam\left(supp(\zeta^\varepsilon_1)\right)\le 2(\frac{1}{\varepsilon})^\beta\cdot\varepsilon\le 2\varepsilon^{1-\beta}.$$ 
	Hence the proof is complete.
\end{proof}

Denote the center of $\zeta^\varepsilon_1$ as
\begin{equation*}
X^\varepsilon_1=\frac{\int_D x\zeta^\varepsilon_1(x)dm(x)}{\int_D \zeta^\varepsilon_1(x)dm(x)}.
\end{equation*}
We are now going to study the location of $X^\varepsilon_1$ as $\varepsilon$ goes to zero. 
\begin{lemma}\label{lem4-9}
	As $\varepsilon\to 0^+$, one has $X^\varepsilon_1 \to \hat{x}$.
\end{lemma}
\begin{proof}
	Assume the lemma is false. We fix $l_0<l_0(\beta)$ and suppose $X^\varepsilon_1\to \tilde{x}$ with $\tilde{x}\ne\hat{x}$. By Lemma \ref{lem4-8} there exists a subsequence $\{\varepsilon_i\}_i^\infty$ satisfying $\varepsilon_i\to 0^+$ as $i\to +\infty$,  such that $supp(\zeta^{\varepsilon_i}_1)\subset B_{2{\varepsilon_i}^{1-\beta}}(\tilde{x})$. We can let $i$ be sufficiently large such that 
	\begin{equation}\label{4-22}
	B_{2{\varepsilon_i}^{1-\beta}}(\tilde{x})\cap\mathcal{P}_{\varepsilon_i}=\varnothing,
	\end{equation}
	where $\mathcal{P}_{\varepsilon_i}$ is the $V_P(\varepsilon_i)$ neighborhood of $\hat{x}$ we have defined before Lemma \ref{lem4-5}. As $i\to+\infty$, there holds $\kappa^{\varepsilon_i}_1>\frac{1}{2}\kappa_0$ and $\kappa^{\varepsilon_i}_{P_1}>\frac{1}{2}\kappa_0$, which yields a contradiction since \eqref{4-22} holds and $\kappa^{\varepsilon_i}_1+\kappa^{\varepsilon_i}_{P_1}>\kappa_0$. So we have proved the lemma.
\end{proof}

In order to show $\zeta^\varepsilon_2$ vanishes as $\varepsilon\to 0^+$, we give the following lemma.
\begin{lemma}\label{lem4-10}
	One has
	\begin{equation}\label{4-23}
	\phi(\hat{x})>\text{max}_{\bar{D}}q.
	\end{equation}
\end{lemma}
\begin{proof}
	The proof is by contradiction. We suppose $\phi(\hat{x})\le\text{max}_{\bar{D}}q$. Recalling that $\phi(x)=\frac{\kappa_0}{4\pi}b(x)+q(x)$ and $b(x)>0\ in\ D$, only two cases may happen:
	\begin{itemize}
	\item[(i)]
	$q(x)$ attains its maximum at $x_0\in D$. One obtains
	$$\phi(x_0)>\text{max}_{\bar{D}}q\ge \phi(\hat{x}),$$
	which contradicts the definition of $\hat{x}$. 
	\item[(ii)]
	$q(x)$ attains its maximum at $x_0\in \partial D$. If $b(x_0)>0$, we can argue as before. While if $b(x_0)=0$ we have 
	$$\phi(\hat{x})=\phi(x_0),$$
	which contradicts the assumption that $\hat{x}\in D$ is unique.
	\end{itemize}
    Thus the proof is completed.
\end{proof} 

Now we can elliminate $\zeta^\varepsilon_2$ from $\zeta^\varepsilon$, which means $\zeta^\varepsilon_1$ is actually $\zeta^\varepsilon$.
\begin{lemma}\label{lem4-11}
	As $\varepsilon\to0^+$, one has $\zeta^\varepsilon=\zeta^\varepsilon_1$. 
\end{lemma}
\begin{proof}
	According to Lemma \ref{lem4-8} and \ref{lem4-9}, for a fixed $l$ satisfying $l_0<l_0(\beta)$, we can take  $\varepsilon$ sufficiently small such that $\varepsilon<\varepsilon(l_0)$ and $supp(\zeta^\varepsilon_1)\subset B_{l/2}(\hat{x})$. That is, $dist\left(supp(\zeta^\varepsilon_1),supp(\zeta^\varepsilon_2)\right)>\frac{l}{2}$. Using the bathtub principle \cite{LL} and preceeding as in Lemma \ref{4-1}, we obtain that for each $x\in supp(\zeta^\varepsilon_2)$, there holds
	\begin{equation*}
	\begin{split}
	\mathcal{K}\zeta^\varepsilon(x)&=\frac{b(x)}{2\pi}\int_D\ln\frac{1}{|x-y|}\zeta^\varepsilon_1(y)d\nu(y)+\frac{b(x)}{2\pi}\int_D\ln\frac{1}{|x-y|}\zeta^\varepsilon_2(y)d\nu(y)+C\delta(\varepsilon)\\
	&\le \frac{b(x)}{2\pi}\int_D\ln(l/2)\zeta^\varepsilon_1(y)d\nu(y)+o_\varepsilon(1)\cdot\kappa_0+C\delta(\varepsilon)\\
	&= o_\varepsilon(1),
	\end{split}
	\end{equation*}
	where $C$ denotes a constant independent of $\varepsilon$. By Lemma \ref{lem4-10} and Lemma \ref{lem4-6}, for each $x\in supp(\zeta^\varepsilon_2)$ and $\varepsilon$ sufficiently small, we have
	$$\mathcal{K}\zeta^\varepsilon(x)+q(x)\le\text{max}_{\bar{D}}q+o_\varepsilon(1)<\phi(\hat{x})+o_\varepsilon(1)<\mu^\varepsilon.$$
	As a result, $\zeta^\varepsilon_2=0$ and $\zeta^\varepsilon=\zeta^\varepsilon_1$. Hence we complete the proof.	
\end{proof}

As Lemma \ref{lem3-11} in Section 3, we are going to study the asymptotic shape of $\zeta^\varepsilon$ by scaling technique. We has defined the rescaled version of $\zeta^\varepsilon$ to be 
$$\xi^\varepsilon(x)=\frac{\varepsilon^2}{\delta(\varepsilon)}\zeta^\varepsilon(X^\varepsilon+\varepsilon x), \ \ x\in D^\varepsilon:=\{x\in\mathbb{R}^2 \ | \ X^\varepsilon+\varepsilon x\in D\}.$$
For convenience, we set $\xi^\varepsilon(x)=0$ if $x\in\mathbb{R}^2\setminus D^\varepsilon$.
 
We denote by $g^\varepsilon$ the symmetric radially nonincreasing Lebesgue-rearrangement of $\xi^\varepsilon$ centered at the origin. The following lemma determines the asymptotic nature of $\zeta^\varepsilon$ in terms of its scaled version $\xi^\varepsilon$.
\begin{lemma}\label{lem4-12}
	Every accumulation points of $\xi^\varepsilon(x)$ as $\varepsilon \to 0^+$, in the weak topology of $L^2$, are radially nonincreasing functions.
\end{lemma}
\begin{proof}
	Up to subsequence we may assume that $\xi^\varepsilon\to\xi^*$ and $g^\varepsilon\to g^*$ weakly in $L^2(\mathbb{R}^2)$ as $\varepsilon\to0^+$. By Riesz's rearrangement inequality, we obtain
	$$\int_{\mathbb{R}^2}\int_{\mathbb{R}^2}\ln\frac{1}{|x-y|}\xi^\varepsilon(x)\xi^\varepsilon(y)dm(x)dm(y)\le\int_{\mathbb{R}^2}\int_{\mathbb{R}^2}\ln\frac{1}{|x-y|}g^\varepsilon(x)g^\varepsilon(y)dm(x)dm(y).$$
	Letting $\varepsilon\to 0^+$, we have
	\begin{equation}\label{4-24}
	\int_{\mathbb{R}^2}\int_{\mathbb{R}^2}\ln\frac{1}{|x-y|}\xi^*(x)\xi^*(y)dm(x)dm(y)\le\int_{\mathbb{R}^2}\int_{\mathbb{R}^2}\ln\frac{1}{|x-y|}g^*(x)g^*(y)dm(x)dm(y).
	\end{equation}
	On the other hand, define 
	$$\zeta^\varepsilon_m(x):=\frac{\delta(\varepsilon)}{\varepsilon^2}g(\varepsilon^{\varepsilon}(x-X^\varepsilon)),\ \ x\in D,$$
	one can easily verify
	$$0\le\zeta^\varepsilon_m\le\text{max}_{\bar{D}}\zeta^\varepsilon \ a.e \ on \ D, \ \ \int_D\zeta^\varepsilon_md\nu=\kappa_0\delta(\varepsilon)+O(\varepsilon)\cdot\delta(\varepsilon).$$
    As $\varepsilon\to 0^+$, a direct calculation yields
	$$(\ln\frac{1}{\varepsilon})^2\mathcal{E}(\zeta^\varepsilon)=\frac{(b(\hat{x}))^3}{4\pi}\int_{\mathbb{R}^2}\ln\frac{1}{|x-y|}\xi^\varepsilon(x)\xi^\varepsilon(y)dm(x)dm(y)+\frac{1}{2}\phi(\hat{x})\kappa_0\cdot\ln\frac{1}{\varepsilon}+A_1(\varepsilon),$$
	and
	$$(\ln\frac{1}{\varepsilon})^2\mathcal{E}(\zeta^\varepsilon_m)=\frac{(b(\hat{x}))^3}{4\pi}\int_{\mathbb{R}^2}\ln\frac{1}{|x-y|}g^\varepsilon(x)g^\varepsilon(y)dm(x)dm(y)+\frac{1}{2}\phi(\hat{x})\kappa_0\cdot\ln\frac{1}{\varepsilon}+A_2(\varepsilon),$$
	where
	$$A_1(\varepsilon)=A_2(\varepsilon)+O(\varepsilon)<\infty$$
	There is an $\left(O(\varepsilon)\cdot\delta^2(\varepsilon)\right)$-perturbation $\zeta^\varepsilon_b$ of $\zeta^\varepsilon_m$ which belongs to $\mathcal{A}_\varepsilon$. Hence
	$$\mathcal{E}(\zeta^\varepsilon)\ge\mathcal{E}(\zeta^\varepsilon_b)\ge \mathcal{E}(\zeta^\varepsilon_m)+O(\varepsilon)\cdot\delta^2(\varepsilon).$$
	Therefore, we conclude 
	$$\int_{\mathbb{R}^2}\int_{\mathbb{R}^2}\ln\frac{1}{|x-y|}\xi^\varepsilon(x)\xi^\varepsilon(y)dm(x)dm(y)\ge\int_{\mathbb{R}^2}\int_{\mathbb{R}^2}\ln\frac{1}{|x-y|}g^\varepsilon(x)g^\varepsilon(y)dm(x)dm(y)+O(\varepsilon). $$
	Again letting $\varepsilon\to 0^+$, one obtains
	\begin{equation}\label{4-25}
	\int_{\mathbb{R}^2}\int_{\mathbb{R}^2}\ln\frac{1}{|x-y|}\xi^*(x)\xi^*(y)dm(x)dm(y)\ge\int_{\mathbb{R}^2}\int_{\mathbb{R}^2}\ln\frac{1}{|x-y|}g^*(x)g^*(y)dm(x)dm(y).
	\end{equation}
	Combining \eqref{4-24} and \eqref{4-25}, we have
	$$\int_{\mathbb{R}^2}\int_{\mathbb{R}^2}\ln\frac{1}{|x-y|}\xi^*(x)\xi^*(y)dm(x)dm(y)=\int_{\mathbb{R}^2}\int_{\mathbb{R}^2}\ln\frac{1}{|x-y|}g^*(x)g^*(y)dm(x)dm(y).$$
	According to the main results in Burchard-Guo \cite{BG}, there exists a translation $\mathcal{T}$ of $\mathbb{R}^2$ such that $\mathcal{T}\zeta^*=g^*$. Notice that 
	$$\int_{\mathbb{R}^2}x\zeta^*(x)dm(x)=\int_{\mathbb{R}^2}x\zeta^*(x)dm(x)=0.$$
	Hence $\zeta^*=g^*$. The proof is complete.	
\end{proof}

\section{Proof of Theorem \ref{thm3}}
Recall that in Theorem \ref{thm3}, we assume $\delta(\varepsilon)\ln\frac{1}{\varepsilon}=o_\varepsilon(1)$ and denote $\mathcal{M}=\{x\in\bar{D}\ |\ q(x)=\text{max}_{\bar{D}} q\}$. The proof in this section is much simpler than the proof of other two theorems. We break down the proof into several lemmas and begin by eliminating the patch part of $\zeta^{\varepsilon,\Lambda}$.
\begin{lemma}\label{lem5-1}
	One has 
	\begin{equation}\label{5-1}
	\mathcal{N}\zeta^{\varepsilon,\Lambda}:=\int_D\ln\frac{diam(D)}{|x-y|}\zeta^{\varepsilon,\Lambda}(y)d\nu(y)\le\kappa_0\delta(\varepsilon)\ln\frac{1}{\varepsilon}+C\delta(\varepsilon),	
	\end{equation}
	where $C>0$ does not depend on $\varepsilon$. As a result, for $\varepsilon$ sufficiently small there holds
	\begin{equation}\label{5-2}
	\mathcal{K}\zeta^{\varepsilon,\Lambda}\le\frac{\text{max}_{\bar{D}}b}{2\pi}\kappa_0\delta(\varepsilon)\ln\frac{1}{\varepsilon}+C\delta(\varepsilon). 
	\end{equation}
	Thus if we take $\Lambda$ sufficiently large, the patch part of $\zeta^{\varepsilon,\Lambda}$ vanishes and \eqref {2-3} can be substituted by
	\begin{equation}\label{5-3}
	\zeta^{\varepsilon,\Lambda}=\frac{\delta(\varepsilon)}{\varepsilon^2}f(\psi^{\varepsilon,\Lambda}){\chi}_{\{x\in D \mid \psi^{\varepsilon,\Lambda}(x)>0\}}.
	\end{equation}
\end{lemma}
\begin{proof}
	By Lemma \ref{lem2-1}, there holds 
	$$\zeta^{\varepsilon,\Lambda}\in\mathcal{A}_{\varepsilon,\Lambda}=\{\zeta\in L^\infty(D)~|~ 0\le \zeta \le \frac{\Lambda\delta(\varepsilon)}{\varepsilon^2}~ \mbox{ a.e. in }D, \int_{D}\zeta(x)d\nu (x)=\kappa_0\delta(\varepsilon) \}.$$ 
	To give a upper bound of $\mathcal{N}\zeta^{\varepsilon,\Lambda}$, we enlarge $\mathcal{A}_{\varepsilon,\Lambda}$ to be $$\mathcal{B}_{\varepsilon,\Lambda}=\{\zeta\in L^\infty(D)~|~ 0\le \zeta b(x) \le \frac{\Lambda\text{max}_{\bar{D}}b\delta(\varepsilon)}{\varepsilon^2}=\frac{C_B\delta(\varepsilon)}{\varepsilon^2}~ \mbox{ a.e. in }D, \int_{D}\zeta(x)d\nu (x)=\kappa_0\delta(\varepsilon) \}.$$
	It is obvious that $\text{sup}_{\mathcal{B}_{\varepsilon,\Lambda}}\mathcal{N}\zeta\ge\text{sup}_{\mathcal{A}_{\varepsilon,\Lambda}}\mathcal{N}\zeta$. By the bathtub principle \cite{LL}, We can calculate just as in Lemma \ref{lem4-1} to obtain
	\begin{equation*}
	\mathcal{N}\zeta^{\varepsilon,\Lambda}\le\text{sup}_{\mathcal{B}_{\varepsilon,\Lambda}}\mathcal{N}\zeta
	\le\kappa_0\delta(\varepsilon)\ln\frac{1}{\varepsilon}+C\delta(\varepsilon),
	\end{equation*}
	where $C$ is a constant independent of $\varepsilon$. Since $C_B=\Lambda\text{max}_{\bar{D}}b$, we can take $\varepsilon<\varepsilon_4(\Lambda)$ such that $$\frac{\text{max}_{\bar{D}}b}{2\pi}\mathcal{N}\zeta^{\varepsilon,\Lambda}+\kappa_0\delta(\varepsilon)||R||_{L^\infty(D\cdot D)}< 2.$$
	By the definition of $\mathcal{K}\zeta^{\varepsilon,\Lambda}$ in Definition \ref{def2}, we obtain \eqref{5-2}. If we fix $\Lambda$ such that
	\begin{equation}\label{5-4}
	f^{-1}(\Lambda)>\text{max}_{\bar{D}}\mathcal{K}\zeta^{\varepsilon,\Lambda}+\text{max}_{\bar{D}}q-\text{min}\mu^{\varepsilon,\Lambda},
	\end{equation}
	one can get
	$$f^{-1}(\Lambda)>\mathcal{K}\zeta^{\varepsilon,\Lambda}(x)+q(x)-\mu^{\varepsilon,\Lambda}\ \ \ \ \forall x\in D.$$
	So it is obvious that
	\begin{equation}\label{5-5}
	|\{x\in D\mid\zeta^{\varepsilon,\Lambda}(x)=\frac{\delta(\varepsilon)\Lambda}{\varepsilon^2}\}|=0.
	\end{equation}
	Hence the patch part of $\zeta^{\varepsilon,\Lambda}$ vanishes and \eqref{5-3} holds.	
\end{proof}	

Now we can fix $\Lambda$ sufficiently large such that \eqref{5-4} holds and abbreviate $(\mathcal{A}_{\varepsilon,\Lambda},\zeta^{\varepsilon,\Lambda},\psi^{\varepsilon,\Lambda}, \mu^{\varepsilon,\Lambda})$ as $(\mathcal{A}_{\varepsilon},\zeta^{\varepsilon}, \psi^{\varepsilon}, \mu^{\varepsilon})$. To prove Theorem \ref{thm3}, we just need to analyze the limiting behavior of $\mu^\varepsilon$ as $\varepsilon$ goes to zero, with no requirement that $f$ satisfies $(\text{H2})$. To illustrate the next lemma, we denote the $l$-neighborhood of $\mathcal{M}$ as
$$\mathcal{M}_l:=\{x\in D|dist(x,\mathcal{M})<l\}.$$
\begin{lemma}\label{lem5-2}
	As $\varepsilon\to0^+$, one has
	\begin{equation}\label{5-6}
	\mu^\varepsilon=\text{max}_{\bar{D}}q+o_\varepsilon(1).
	\end{equation}	
\end{lemma}
\begin{proof}
	According to Lemma \ref{lem5-1}, as $\varepsilon\to0^+$ there holds $$\mu^\varepsilon\le\text{max}_{\bar{D}}(\mathcal{K}\zeta^\varepsilon+q)\le\text{max}_{\bar{D}}q+o_\varepsilon(1).$$ 
	Hence we only have to prove
	\begin{equation}\label{5-7}
	\liminf\limits_{\varepsilon\to0^+}\mu^\varepsilon\ge \text{max}_{\bar{D}}q
	\end{equation}
	Then we can argue by contradiction. Suppose there exists an $a>0$ and a subsequence $\{\varepsilon_i\}_i^\infty$ satisfying $\varepsilon_i\to 0^+$ as $i\to +\infty$, such that $\mu^{\varepsilon_i}\le\text{max}_{\bar{D}}q-a$ when $i$ is sufficiently large. Take $l_1>0$ sufficiently small, such that $\text{inf}_{\mathcal{M}_{l_1}}q\ge\text{max}_{\bar{D}}q-\frac{a}{2}$. Suppose $\mathcal{M}_{l_1}\subset supp(\zeta^{\varepsilon_i})$ and we obtain
	$$\kappa_0\delta(\varepsilon_i)=\int_D\zeta^{\varepsilon_i} d\nu\ge\int_{\mathcal{M}_{l_1}}\zeta^{\varepsilon_i} d\nu=\int_{\mathcal{M}_{l_1}}\frac{\delta(\varepsilon_i)}{{\varepsilon_i}^2}f(\mathcal{K}\zeta^{\varepsilon_i}+q-\mu^{\varepsilon_i})d\nu.$$
	By \eqref{5-2}, we can take $i$ sufficiently large such that $|\mathcal{K}\zeta^{\varepsilon_i}|\le\frac{a}{4}$. Then for each $x\in\mathcal{M}_{l_1}$, there holds
	$$f\left((\mathcal{K}\zeta^{\varepsilon_i}+q)(x)-\mu^{\varepsilon_i}\right)\ge\frac{a}{4},$$
	and we obtain $$\kappa_0\delta(\varepsilon_i)\ge\int_{\mathcal{M}_{l_1}}\frac{\delta(\varepsilon_i)}{{\varepsilon_i}^2}f(\mathcal{K}\zeta^{\varepsilon_i}+q-\mu^{\varepsilon_i})d\nu\ge\frac{a\delta(\varepsilon_i)}{4{\varepsilon_i}^2}|\mathcal{M}_{l_1}|,$$
	which yields a contradiction as $\varepsilon_i\to0^+$. Thus we assert $\mathcal{M}_{l_1}\not\subset supp(\zeta^{\varepsilon_i})$. So we can take a sequence $\{x^i\}$ satisfying $x^i\in\mathcal{M}_{l_1}\setminus supp(\zeta^\varepsilon)$ and $q(x^i)\ge\text{max}_{\bar{D}}q-\frac{a}{2}$, such that
	$$\mu^{\varepsilon_i}\ge \mathcal{K}\zeta^{\varepsilon_i}(x^i)+q(x^i)>-\frac{a}{4}+\text{max}_{\bar{D}}q-\frac{a}{2}=\text{max}_{\bar{D}}q-\frac{3a}{4},$$
	which contradicts the assumpption $\mu^{\varepsilon_i}\le\text{max}_{\bar{D}}q-a$. Thus we complete the proof.	
\end{proof}

In the last, we study the localization of $supp(\zeta^\varepsilon)$ as $\varepsilon$ goes to zero.
\begin{lemma}\label{lem5-3}
	For any $l>0$, we have $supp(\zeta^\varepsilon)\in\mathcal{M}_l$ if $\varepsilon$ is sufficiently small.
\end{lemma}
\begin{proof}
The proof is by contradiction. We suppose there exists an $s>0$, a sequence $\{x^i_0\}$ in $D$ and a subsequence $\{\varepsilon_i\}_i^\infty$ satisfying $\varepsilon_i\to 0^+$ as $i\to +\infty$, such that $x^i_0\in supp(\zeta^{\varepsilon_i})$ but $x^i_0\not\in \mathcal{M}_s$. By the continuity of $q$, we have $$\limsup_{n\to\infty}q(x^i_0)<\text{max}_{\bar{D}}q.$$
Hence when $i$ is sufficiently large, there holds
$$\mathcal{K}\zeta^{\varepsilon_i}(x^i_0)+q(x^i_0)<\mu^{\varepsilon_i}=\text{max}_{\bar{D}}q+o_\varepsilon(1).$$
We get a contradiction and hence complete the proof.
\end{proof}


\begin{thebibliography}{99}
	
	\bibitem{AR}
	A. Ambrosetti and P. Rabinowitz, Dual variational methods in critical point theory and applications, \textit{J. Funct. Anal}, 14(1973),349-381.
	
	\bibitem{AS}
	A. Ambrosetti and M. Struwe, Existence of steady vortex rings in an ideal fluid, \textit{Arch. Rational Mech. Anal}, 108(2)(1989), 97--109.  
	
	\bibitem{A}
	V. I. Arnold, Mathematical methods of classical mechanics, \textit{Graduate Texts in Mathematics, Vol. 60.} Springer, New York, 1978.
		
	\bibitem{AK}
	V. I. Arnold and B. A. Khesin, Topological methods in hydrodynamics,
	\textit{Applied Mathematical Sciences, Vol. 125.} Springer, New
	York, 1998.
	
	\bibitem{BF1}
	M. S. Berger and L. E. Fraenkel, A global theory of steady vortex rings in an ideal fluid, \textit{Acta Math.},
	132(1974), 13--51.
	
	\bibitem{BF2}
	M. S. Berger and L. E. Fraenkel, Nonlinear desingularization in certain free-boundary problems, \textit{Comm. Math. Phys.},
	77(1980), 149--172.
	
	\bibitem{BG}
	A. Burchard and Y. Guo, Compactness via symmetrization, \textit{J. Funct. Anal.}, 214(1)(2004), 40-73.
	
	\bibitem{B}
	G. R. Burton, Global nonlinear stability for steady ideal fluid flow in bounded planar domains, \textit{Arch. Ration. Mech. Anal.}, 176(2005), 149-163.
	
	\bibitem{CWZ}
	D. Cao, J. Wan and W. Zhan, Desingularization of vortex rings in 3 dimensional Euler flows: with swirl, arXiv:1909.00355.
	
	\bibitem{CW3}
	D. Cao, G. Wang and W. Zhan, Steady vortex patches near a nontrivial irrotational flow, \textit{Sci China Math,} 63(2020), https://doi.org/10.1007/s11425-018-9495-1.
	
	\bibitem{CZZ}
	D. Cao, W. Zhan and C. Zou, On desingularization of steady vortex in the lake equations, arXiv: 1911.01037.
	
	\bibitem{D1}
	J. Dekeyser, Desingularization of a steady vortex pair in the lake equation, arXiv:1711.06497.
	
	\bibitem{D2}
	J. Dekeyser, Asymptotic of steady vortex pair in the lake equation, \textit{SIAM J. Math. Anal.}, 51 (2019), no. 2, 1209-1237.
		
	\bibitem{DV}
	S. de Valeriola and J. Van Schaftingen, Desingularization of vortex rings and shallow water vortices by semilinear elliptic problem,
	\textit{ Arch. Ration. Mech. Anal.}, 210(2)(2013), 409--450.
	
	\bibitem{FT}
	A. Friedman and B. Turkington, Vortex rings: existence and asymptotic estimates, \textit{Trans. Am. Math. Soc}, 268(1)(1981), 1--37.
		
	\bibitem{LYY} G. Li, S. Yan and J. Yang, An elliptic problem related to planar vortex pairs,
	\textit{ SIAM J. Math. Anal.}, 36 (2005), 1444--1460.
	
	\bibitem{LL} E. H. Lieb and M. Loss, Analysis, Second edition,\textit{ Graduate Studies in Mathematics, Vol. 14}. American Mathematical Society, Providence, RI (2001).

	\bibitem{SV}
	D. Smets and  J. Van Schaftingen, Desingulariation of vortices for
	the Euler equation,
	\textit{ Arch. Ration. Mech. Anal.},  198(2010),   869--925.	
	
	\bibitem{T}
	B. Turkington, On steady vortex flow in two dimensions. I, II,
	\textit{Comm. Partial
		Differential Equations}, 8(1983), 999--1030, 1031--1071.
	
	\bibitem{YJ}
	J. Yang, Existence and asymptotic behavior in planar vortex theory, 
	\textit{Math. Models Methods Appl. Sci.}, 1(4)(1991), 461-475.
	

\end{thebibliography}
\end{document}